\documentclass[11pt]{amsart}
\usepackage[
top=2.7cm, bottom=2.5cm, inner=2.5cm, outer=2.4cm,
marginparwidth=2.3cm
]{geometry}

\usepackage{amssymb,latexsym, amsmath, amsxtra, mathrsfs, bm}

\usepackage[dvips]{graphics}
\usepackage{xypic,xcolor}
\usepackage{subcaption}

\usepackage{verbatim}
\usepackage[abs]{overpic}
\usepackage{hyperref}
\usepackage{mathtools}

\DeclareMathAlphabet{\mathbbold}{U}{bbold}{m}{n}

\allowdisplaybreaks[3]

\theoremstyle{plain}
        \newtheorem{theorem}{Theorem}[section]
        \newtheorem*{theorem*}{Theorem}
        \newtheorem*{conj*}{Conjecture}
        \newtheorem{lemma}[theorem]{Lemma}
        \newtheorem{prop}[theorem]{Proposition}
        \newtheorem{cor}[theorem]{Corollary}

        \newtheorem{thmx}{Theorem}

\theoremstyle{definition}
        \newtheorem{definition}[theorem]{Definition}
        \newtheorem{rem}[theorem]{Remark}

         \newtheorem*{assumptions}{The Assumptions}

\theoremstyle{remark}
        \newtheorem*{remark}{Remark}

\numberwithin{equation}{section}
\numberwithin{theorem}{section}
\numberwithin{table}{section}
\numberwithin{figure}{section}

\providecommand{\defn}[1]{\emph{#1}}

\renewcommand{\leq}{\leqslant}

\renewcommand{\geq}{\geqslant}

\newcommand{\diam}  {\operatorname{diam}}

\newcommand{\id} {\operatorname{id}}

\newcommand{\card} {\operatorname{card}}

\newcommand{\R}{\mathbb{R}}
    
\newcommand{\C}{\mathbb{C}}

\newcommand{\N}{\mathbb{N}}      
\newcommand{\Z}{\mathbb{Z}}      

\providecommand{\abs}[1]{\lvert#1\rvert}
\providecommand{\Absbig}[1]{\bigl\lvert#1\bigr\rvert}

\renewcommand{\:}{\colon}

\newcommand{\post}{\operatorname{post}}

\renewcommand{\=}{\coloneqq}

\newcommand{\cA}{\mathcal{A}}
\newcommand{\cB}{\mathcal{B}}

\newcommand{\cI}{\mathcal{I}}
\newcommand{\cJ}{\mathcal{J}}

\renewcommand{\cL}{\mathcal{L}}
\newcommand{\cM}{\mathcal{M}}

\newcommand{\cP}{\mathcal{P}}

\newcommand{\cS}{\mathcal{S}}

\newcommand{\cU}{\mathcal{U}}
\newcommand{\cV}{\mathcal{V}}
\newcommand{\cW}{\mathcal{W}}

\newcommand{\cY}{\mathcal{Y}}

\newcommand{\fC}{\mathfrak{C}}

\newcommand{\fF}{\mathfrak{F}}
\newcommand{\fG}{\mathfrak{G}}

\newcommand{\fX}{\mathfrak{X}}

\newcommand{\hU}{\widehat{U}}

\newcommand{\tU}{\widetilde{U}}

\newcommand{\tx}{\widetilde{x}}
\newcommand{\ty}{\widetilde{y}}

\newcommand{\tgamma}{\widetilde{\gamma}}

\newcommand{\Bouquet}{\mathcal{B}}
\newcommand{\bouquet}{\operatorname{B}}
\newcommand{\closure}[1]{\overline{#1}}
\newcommand{\diff}{{\,\mathrm{d}}}
\newcommand{\Round}{\operatorname{Round}}

\renewcommand{\top}{\operatorname{top}}

\theoremstyle{definition}
    \newtheorem*{addassumptions}{The Additional Assumptions}

\begin{document}

\title[Weak expansion properties]{Weak expansion properties and a large deviation principle for coarse expanding conformal systems}
\author{Zhiqiang~Li \and Hanyun~Zheng}

\address{Zhiqiang~Li, School of Mathematical Sciences \& Beijing International Center for Mathematical Research, Peking University, Beijing 100871, China.}
\email{zli@math.pku.edu.cn}

\address{Hanyun~Zheng, School of Mathematical Sciences, Peking University, Beijing 100871, China.}
\email{1900013001@pku.edu.cn}



\begin{abstract}
	Coarse expanding conformal systems were introduced by P.~Ha\"issinsky and K.~M.~Pilgrim to study the essential dynamical properties of certain rational maps on the Riemann sphere in complex dynamics from the point of view of Sullivan's dictionary.
	
    In this paper, we prove that for a metric coarse expanding conformal system $f\:(\fX_1,X)\rightarrow (\fX_0,X)$ with repellor $X$, the map $f|_X\:X\rightarrow X$ is asymptotically $h$-expansive. Moreover, we show that $f|_X$ is not $h$-expansive if there exists at least one branch point in the repellor. We also prove that $f|_X$ is forward expansive when there is no branch point in the repellor. 
    	
    Our study of expansion properties does not rely on symbolic codings or Markov partitions, and exploits directly the geometric data. 
    
    As a consequence of asymptotic $h$-expansiveness, for $f|_X$ and each real-valued continuous potential on $X$, there exists at least one equilibrium state. For such maps, if some additional assumptions are satisfied, we can furthermore establish a level-$2$ large deviation principle for iterated preimages, followed by an equidistribution result.
\end{abstract}

\maketitle

\tableofcontents

\section{Introduction}\label{s:Introduction}
\subsection{Background and motivation}
Expansiveness is a well-known condition in the study of dynamical systems. In the context of continuous maps on compact metric spaces, there are two weaker notions of expansion, namely \defn{$h$-expansiveness} and \defn{asymptotic $h$-expansiveness}, introduced by R.~Bowen and M.~Misiurewicz, respectively. Both notions play important roles in the study of dynamical systems (see for example, \cite{DN05,DM09,Bur11}). Moreover, any smooth map on a compact Riemannian manifold is asymptotically $h$-expansive \cite{Buz97}.

In complex dynamics, the first-named author studied $h$-expansiveness and asymptotic $h$-expansiveness for expanding Thurston maps \cite{Li15}. In this paper, we consider a more general context in complex dynamics. The dynamical systems that we study in this paper are called \emph{coarse expanding conformal (CXC) systems}, introduced by P.~Ha\"issinsky and K.~M.~Pilgrim \cite{HP09}. CXC systems are a class of topological systems defined in an axiomatic way, generalizing a wide range of maps, including some non-uniformly expanding rational maps on the Riemann sphere $\widehat\C$ and expanding Thurston maps without periodic critical points studied by M.~Bonk and D.~Meyer \cite{BM10,BM17} dating back to W.~P.~Thurston's topological characterization of postcritially finite rational maps. Other examples of CXC systems include
\begin{enumerate}
	\smallskip
	\item[\textnormal{(1)}] hyperbolic, subhyperbolic, and semihyperbolic rational maps on the Riemann sphere $\widehat\C$, acting on their Julia sets (equipped with spherical metric);
	\smallskip
	\item[\textnormal{(2)}] Latt\`es type uniformly quasiregular maps on Riemannian manifolds;
	\smallskip
	\item[\textnormal{(3)}] smooth expanding partial self-covering maps of Riemannian manifolds. More precisely, let $\fX_0$ be a connected complete Riemannian manifold and $\fX_1\subseteq\fX_0$ an open submanifold with finitely many connected components which is compactly contained in $\fX_0$, then an expanding $C^1$ covering map $f\:\fX_0\rightarrow\fX_1$ is CXC.
\end{enumerate}
We refer the reader to \cite[Sections~2.3, 4.2, and 4.4]{HP09}, and  \cite{CJY94,CT11,HP10,JZ09} for more information regarding examples listed above.

The study of CXC systems is motivated by complex dynamics and is devoted to revealing the analogy between geometric group theory and complex dynamics. Vast literature concerning analogies between Kleinian groups and rational maps is devoted to Sullivan's dictionary, and an enlargement of the dictionary is suggested in the study of CXC systems (see \cite{HP09}).

In this paper we establish the asymptotic $h$-expansiveness of each metric CXC system $f\:(\fX_1,X)\rightarrow(\fX_0,X)$ with repellor $X$, more precisely, of the map $f|_X$. We also study expansiveness and $h$-expansiveness of metric CXC systems. The motivation to investigate asymptotic $h$-expansiveness comes from the theory of thermodynamic formalism. Investigation on concepts such as topological entropy, measure-theoretic entropy, topological pressure, and measure-theoretic pressure is a main theme in thermodynamic formalism. The maximizing measures of measure-theoretic entropy and measure-theoretic pressure are known as measures of maximal entropy and equilibrium states, respectively. The existence, uniqueness, and various other properties of equilibrium states have been studied in many contexts (see for example, \cite{Bo75, Ru89, Pr90, KH95, MauU03, BS03, Ol03, Yu03, PU10, MayU10}).

M.~Misiurewicz showed that asymptotic $h$-expansiveness of a map $f$ guarantees that the measure-theoretic entropy $\mu\mapsto h_\mu(f)$ is upper semi-continuous \cite{Mi76} with respect to the weak$^*$ topology. It implies that for each real-valued continuous function $\phi$, the measure-theoretic pressure $\mu\mapsto P_\mu(f,\phi)$ is upper semi-continuous, which guarantees the existence of equilibrium states.

Once we obtain the existence of equilibrium states, it is natural to ask questions regarding uniqueness and other properties. In the context of CXC systems, the uniqueness problem remains unsettled for general continuous potentials $\phi$. However, thermodynamic formalism is established in related systems for potentials under some mild regularity assumptions by Das et al.~\cite{DPTUZ21} building upon prior works of P.~Ha\"issinsky and K.~M.~Pilgrim \cite{HP09}. In such a context, they established the uniqueness of equilibrium states for H\"older continuous potentials and CXC systems under additional assumptions. The first-named author has been informed that similar results on equilibrium states have also been obtained independently by P.~Ha\"issinsky.

Based on the existence and uniqueness of equilibrium states, it is natural to investigate how iterated preimages are distributed with respect to such measures. We establish a \defn{level-$2$ large deviation principle}, and consequently, the equidistribution of iterated preimages with respect to the unique equilibrium state. We use a variant of Y.~Kifer's result \cite{Ki90} formulated by H.~Comman and J.~Rivera-Letelier \cite{CRL11}, recorded in Theorem~\ref{t:LD, Kifer and Juan}. We refer the reader to \cite[Section~2.5]{CRL11} and the references therein for a more systematic introduction to the theory of large deviation principles. The equidistribution result follows from the corresponding level-$2$ large deviation principle.

\subsection{Main results} In this paper, we study some weak expansion properties of metric CXC systems and apply them to prove some statistical properties. The main results are listed below.

First, we establish the asymptotic $h$-expansiveness for metric CXC systems (see Subsection~\ref{ss:metric CXC} for the definition).
\begin{thmx}[Asymptotic $h$-expansiveness]\label{t:Intro asym h-expans}
    Let $f\:(\fX_1,X)\rightarrow(\fX_0,X)$ be a metric CXC system with repellor $X$. Then $f|_X\:X\rightarrow X$ is asymptotically $h$-expansive. 
\end{thmx}

Then, we study forward expansiveness and $h$-expansiveness of metric CXC systems.
\begin{thmx}\label{t:Intro_expansive and h-expansive}
	Let $f\:(\fX_1,X)\rightarrow(\fX_0,X)$ be a metric CXC system. Then, the following statements are equivalent:
	\begin{itemize}
		\smallskip
		\item [\textnormal{(i)}] There is no branch point in $X$.
		\smallskip
		\item [\textnormal{(ii)}] $f|_X$ is forward expansive.
		\smallskip
		\item [\textnormal{(iii)}] $f|_X$ is $h$-expansive.
	\end{itemize}
\end{thmx}



As a consequence of Theorem~\ref{t:Intro asym h-expans}, we establish the existence of equilibrium states for the map $f|_X$ and each continuous potential $\phi$.
\begin{thmx}[Existence of equilibrium states]\label{t:Intro existence equilibrium}
    Let $f\:(\fX_1,X)\rightarrow(\fX_0,X)$ be a metric CXC system with repellor $X$. Then, for each real-valued continuous function $\phi\:X\rightarrow\R$, there exists at least one equilibrium state for the map $f|_X$ and potential $\phi$. In particular, there exists at least one measure of maximal entropy for $f|_X$.
\end{thmx}

Applying some additional assumptions (see Subsection~\ref{ss:LD addassumptions}), we furthermore establish a level-$2$ large deviation principle (see Subsection~\ref{ss:level-2 LD} for the definition) for metric CXC systems. We refer the reader to Section~\ref{s:Notation} for the notations in the statement of the following theorem.
\begin{thmx}[Level-$2$ large deviation principle]\label{t:Intro_LD}
    Let $f\:(\fX_1,X)\rightarrow(\fX_0,X)$ be a metric CXC system and denote $f_X\=f|_X$. Assume that the following conditions are satisfied:
\begin{enumerate}
	\smallskip
      \item[\textnormal{(1)}] $\fX_0$ and $\fX_1$ are strongly path-connected.
    \smallskip
      \item[\textnormal{(2)}] The branch set $B_f$ is finite. 
\end{enumerate}
    For each $n\in\N$, let $W_n\:X\rightarrow\cP(X)$ be the continuous function defined by
    \begin{equation*}
        W_n(x)\=\frac{1}{n}\sum_{i=0}^{n-1}\delta_{f^i(x)},
    \end{equation*}
    and denote $S_n\psi(x) \=\sum_{i=0}^{n-1}\psi(f^i(x))$ for $\psi\in C(X)$ and $x\in X$. 
    Let $\phi$ be a real-valued H\"older continuous function on $X$, and define
    \begin{equation}   \label{e:I^phi}
    I^\phi(\mu) \=\begin{cases}
                P(f_X,\phi)-\int\!\phi \diff\mu-h_\mu(f_X) &\mbox{ if }\mu\in\cM(X,f_X);\\
        +\infty &\mbox{ if } \mu\in\cP(X)\setminus \cM(X,f_X).
    \end{cases}
    \end{equation}
    Then for each sequence $\{x_n\}_{n\in\N}$ of points in $X$, the sequence $\{\Omega_n\}_{n\in\N}$ of Borel probability measures on $\cP(X)$ given by
\begin{equation*}
    \Omega_n\=
    \sum\limits_{y\in f^{-n}(x_n)}
    \frac{\deg(f^n;y)\cdot\exp{(S_n\phi(y))}}
    {\sum_{z\in f^{-n}(x_n)}\deg(f^n;z)\cdot\exp{(S_n\phi(z))}}\delta_{W_n(y)}
\end{equation*}
    satisfies a level-$2$ large deviation principle with rate function $I^\phi$, and it converges in the weak$^*$ topology to the Dirac measure supported on the unique equilibrium state for $f$ and potential $\phi$. Furthermore, for each convex open subset $\fG$ of $\cP(X)$ containing some invariant measure, we have
    \begin{equation}\label{e:LD}
        \lim_{n\to+\infty}\frac{1}{n}\log(\Omega_n(\fG))
        =\lim_{n\to+\infty}\frac{1}{n}\log(\Omega_n(\closure\fG))
        =-\inf_\fG I^\phi=-\inf_{\closure\fG} I^\phi.
    \end{equation}
\end{thmx}

Under proper assumptions, there exists a unique equilibrium state for the map $f|_X$ and each H\"older continuous potential due to \cite[Theorem~1.1]{DPTUZ21}. A consequence of the level-$2$ large deviation principle is the equidistribution of iterated preimages with respect to the equilibrium state (see Corollary~\ref{c:Equidistribution}).

\begin{remark}
As shown in the recent work \cite{LS24} of the first-named author and Xianghui~Shi, for expanding Thurston maps, the absence of periodic critical points may not be essential for level-2 large deviation principles, even though the proof in the general case is much more involved relying on the construction of so-called \defn{subsystems}. Subsystems rely on Markov partitions for expanding Thurston maps. It is therefore also interesting to investigate level-2 large deviation principles as in Theorem~\ref{t:Intro_LD} without assuming the [{\bf Deg}] axiom on $f$ and without using Markov partitions.
\end{remark}

 Our results in this paper partially extend results in \cite{HP09, BM10, BM17, Li15, DPTUZ21}.
 
 Note that each expanding Thurston map $f \: S^2 \rightarrow S^2$ without periodic critical points is a metric CXC system (if we equip the sphere with a visual metric) with all branch points in $X$. So Theorems~\ref{t:Intro asym h-expans} and~\ref{t:Intro_expansive and h-expansive} strengthen Theorem~1.1 in \cite{Li15} in the case where $f$ has no periodic critical points. In contrast, if an expanding Thurston map has a periodic critical point, then it is not a metric CXC system (or even a topological CXC system, see Subsection~\ref{ss:CXC}), and is known to be not asymptotically $h$-expansive (see \cite[Theorem~1.1]{Li15}). In the special case where the potential vanishes everywhere, the existence and uniqueness of the measure of maximal entropy and the corresponding equidistribution results were obtained by M.~Bonk and D.~Meyer \cite{BM10, BM17} for expanding Thurston maps, and by P.~Ha\"issinsky and K.~M.~Pilgrim \cite{HP09} for CXC systems.

 One major difference between a Thurston map and a CXC system is that the former is defined on a topological $2$-sphere, while the latter considers more general topological spaces, and that the former requires the set of forward iterated images of critical points to be finite, while the latter does not require such a strong assumption. This so-called \emph{postcritical finiteness} assumption is crucial in the work of M.~Bonk and D.~Meyer \cite{BM10, BM17} for the construction of certain forward invariant Jordan curves that induce Markov partitions for (sufficiently high iterations of) an expanding Thurston map (see also the work of J.~W.~Cannon, W.~J. Floyd, and W.~R.~Parry \cite{CFP07}). The proof of \cite[Theorem~1.1]{Li15} depends heavily on the combinatorial structures induced by the forward invariant Jordan curves and the corresponding \defn{Markov partitions}, which are not available in the general setting of CXC systems. Our strategy in this paper builds upon the construction of \emph{bouquets} (see Definition~\ref{d:bouquet}) from the more flexible \emph{good open covers} of P.~Ha\"issinsky and K.~M.~Pilgrim \cite{HP09}.

 The main technical device $E_m(U_0,\,\dots,\, U_{n-1};U_n)$ (see (\ref{e:Em})) is now defined using good open sets $U_0,\,\dots,\,U_n\in\cU_m$, which differs from similar constructions in \cite{Li15} relying on Markov partitions. The lack of Markov partitions results in significant technical difficulties. The combinatorial data of expanding Thurston maps plays an important role in \cite{Li15}. In \cite{Li15}, the combinatorial structure of so-called flowers, the bi-Lipschitz property with respect to visual metrics of an expanding Thurston map, and the finiteness of the postcritical set, etc., are widely used. On the other hand, in the current paper, we exploit directly the geometric data of CXC systems, for example, the roundness distortion properties and local comparability of the good open sets as well as the so-called bouquets. 
 
 We believe that our geometric approach, which does not rely on symbolic codings or Markov partitions, can be applied in other situations and inspire further research.

 \subsection{Structure of the paper} We will now give a brief description of the structure of this paper. 
 
 In Section~\ref{s:Notation}, we fix some notations in this paper.
 
 In Section~\ref{s:Preliminaries}, we review the definitions of topological CXC systems and metric CXC systems, along with some basic properties recorded from \cite{HP09}. We direct the reader to \cite[Chapter~2]{HP09} for a more detailed study.

 In Section~\ref{s:assumption}, we state the assumptions under which we establish our results. We will repeatedly refer to the assumptions as \defn{the Assumptions in Section~\ref{s:assumption}}. Note that the Assumptions in Section~\ref{s:assumption} serve as a summary of the assumptions adopted in the definition of CXC systems. We summarize them for clarity and convenience.

 In Section~\ref{s:asymptotic h-expansiveness}, we investigate the weak expansion properties of metric CXC systems.
We first introduce some basic concepts in Subsection~\ref{ss:asymptotic h-expans_basic concept}. We review the notion of topological conditional entropy $h(g|\eta)$ of a continuous map $g\:Z\rightarrow Z$ given an open cover $\eta$ of $Z$, and the notion of topological tail entropy $h^*(g)$. We then define $h$-expansiveness and asymptotic $h$-expansiveness using these notions.
In Subsection~\ref{ss:asym h-expans_lemmas}, we prove several lemmas that will be used in the proof of the asymptotic $h$-expansiveness.
Subsection~\ref{ss:pf thm asym & not h-expans} consists of the proofs of Theorems~\ref{t:Intro asym h-expans}, Theorem~\ref{t:Intro_expansive and h-expansive}. 

In Section~\ref{s:equilibrium states}, we show the existence of equilibrium states for a metric CXC system and a continuous potential, which is a consequence of the asymptotic $h$-expansiveness. This section contains the proof of Theorem~\ref{t:Intro existence equilibrium}.

In Section~\ref{s:LD}, with some additional assumptions, we study a level-$2$ large deviation principle and an equidistribution result for iterated preimages. 
In Subsection~\ref{ss:level-2 LD}, we briefly review a level-$2$ large deviation principle in our context, and record in Theorem~\ref{t:LD, Kifer and Juan} a theorem of Y.~Kifer \cite{Ki90}, reformulated by H.~Comman and J.~Rivera-Letelier \cite{CRL11}, on level-$2$ large deviation principles. This result will be applied later to our context.
In Subsection~\ref{ss:LD addassumptions}, we state the additional assumptions, which will be repeatedly referred to as \defn{the Additional Assumptions in Subsection~\ref{ss:LD addassumptions}}. 
Subsection~\ref{ss:pf of LD} consists of the proof of Theorem~\ref{t:Intro_LD}.

\subsection*{Acknowledgments} 
The first-named author would like to thank Peter~Ha\"issinsky and Juan Rivera-Letelier for insightful discussions. The authors would like to thank the anonymous referees for their valuable suggestions. The authors were partially supported by NSFC Nos.~12471083, 12101017, 12090010, 12090015, BJNSF No.~1214021, and Peking University Funding No.~7101303303, .


\section{Notation}\label{s:Notation}
In this section, we set up some notations for this paper. Let $\Z$ be the set of integers, $\N\=\{1,\,2,\,3,\,\dots\}$ be the set of positive integers, and $\N_0\=\{0\}\cup\N$. The symbols $\log$ and $\exp$ denote the logarithm and exponential function to the base $e$, respectively. The symbol $\card(A)$ denotes the cardinality of a set $A$. For $x\in\R$, define $\lceil x\rceil\=\min\{n\in\Z : n\geq x\}$ and $\lfloor x\rfloor\=\max\{n\in\Z : n\leq x\}$.

For a map $f\:X\rightarrow Y$, we denote the restriction of $f$ to a subset $Z$ by $f|_Z$.

For a metric space $(X,d)$, the distance between $a$ and $b$ is denoted by $\abs{a-b}$ when the metric is clear from the context. The diameter of a subset $A$ is denoted by $\diam A\=\sup\{d(a,b) : a,b\in A\}$. The closure of a subset $A$ is denoted by $\closure{A}$.

Let $C(X)$ be the space of continuous functions from $X$ to $\R$. For $\alpha\in(0,1]$, $C^{0,\alpha}(X)$ is the space of $\alpha$-H\"older continuous function from $(X,d)$ to $\R$. Let $\cP(X)$ be the set of Borel probability measures on $X$. For $g: X\rightarrow X$, let $\cM(X,g)$ be the set of $g$-invariant Borel probability measures on $X$. For a point $x\in X$, we denote the Dirac measure supported on $\{x\}$ by $\delta_x$.

Consider a map $f \: X \rightarrow X$ on a set $X$. 
The inverse map of $f$ is denoted by $f^{-1}$. We write $f^n$ for the $n$-th iterate of $f$, and $f^{-n}\= (f^n)^{-1}$, for $n\in \N$. 
We set $f^0 \define \id{X}$, the identity map on $X$. For a real-valued function $\varphi \: X \rightarrow \R$, we write
\begin{equation}    \label{eq:def:Birkhoff average}
	S_n \varphi(x) \= \sum_{j=0}^{n-1} \varphi \bigl( f^j(x) \bigr)
\end{equation}
for $x\in X$ and $n\in \N_0$.

\section{Preliminaries} \label{s:Preliminaries}
We review the definitions of finite branched covering (abbr: FBC) maps and metric coarse expanding conformal (abbr: metric CXC) systems, and we give some basic properties of metric CXC systems.

\subsection{Finite branched covering maps} \label{ss:Finitebranchedcovering}

We define here the notion of finite branched covering maps. FBC maps generalize some rational maps on the Riemann sphere, capturing their essential topological properties. Here we refer to \cite[Section~2.1]{HP09}, where the reader can find a more detailed introduction.

Suppose that $X$ and $Y$ are locally compact Hausdorff spaces, and let $f\:X\rightarrow Y$ be a finite-to-one continuous map. The \defn{degree} of $f$ is 
\begin{equation}
    \deg(f)\=\sup\left\{\card\left(f^{-1}(y)\right) : y\in Y\right\}.
\end{equation}
For $x\in X$, the \defn{local degree} of $f$ at $x$ is
\begin{equation}
    \deg(f;x)\=\inf\limits_{U}\left\{\sup\left\{\card\left(f^{-1}(z)\cap U\right): z\in f(U)\right\}\right\},
\end{equation}
where $U$ ranges over all neighborhoods of $x$.

\begin{definition}[Finite branched covering maps]\label{d:FBC}
Suppose that $X$ and $Y$ are locally compact Hausdorff spaces, and let $f\:X\rightarrow Y$ be a finite-to-one continuous map.
The map $f$ is a \defn{finite branched covering} (abbr: FBC) map if $\deg(f) < +\infty$ and the following conditions are satisfied:
\begin{enumerate}
    \smallskip
    \item[\textnormal{(i)}] The identity
    \begin{equation*}
        \sum\limits_{x\in f^{-1}(y)}\deg(f;x)=\deg(f)
    \end{equation*}
    holds for all $y\in Y$.
    
    \smallskip
    \item[\textnormal{(ii)}]
    For every $x_0 \in X$, there are compact neighborhoods $U$ and $V$ of $x_0$ and $f(x_0)$, respectively, such that the identity
    \begin{equation*}
        \sum\limits_{x\in U,\,f(x)=z}\deg(f;x)=\deg(f;x_0)
    \end{equation*}
    holds for all $z\in V$.
    
\end{enumerate}
\end{definition}

\begin{remark}
We note two consequences of (ii): 
\begin{enumerate}
    \smallskip
    \item[\textnormal{(a)}] Let $W\=f^{-1}(V)\cap U$. The restriction $f|_W\:W\rightarrow V$ is proper and surjective.
    \smallskip
    \item[\textnormal{(b)}] $f^{-1}\left(f(x_0)\right)\cap U=\{x_0\}$.
\end{enumerate} 
\end{remark}

Here are some basic concepts related to FBC maps.
\begin{definition}[Principal value, branch set, and branch value]
   For an FBC map $f\:X\rightarrow Y$, a point $y\in Y$ is a \defn{principal value} if $\card\left(f^{-1}(y)\right)=\deg(f)$. The \defn{branch set} is defined as $B_f\=\{x\in X : \deg(f;x)>1\}$. The set of \defn{branch values} is defined as $V_f\=f(B_f)$.
\end{definition}

We give some basic properties of FBC maps below. Further details can be found in \cite[Section~2.1]{HP09}.

The composition of two FBC maps is a FBC map. The degrees and local degrees multiply under composition. Condition~(ii) in Definition~\ref{d:FBC} implies that if a sequence of points $\{x_n\}$ converges to $x_0$, then $\deg(f;x_n)\leq\deg(f;x_0)$ for all sufficiently large $n\in\N$. It follows that the branch set $B_f$ is closed. The set of principal values is $Y\setminus V_f$.

We record two properties of FBC maps below. See \cite[Lemmas~2.1.2 and~2.1.3]{HP09}.
\begin{prop}[Ha\"issinsky \& Pilgrim \cite{HP09}] \label{p:FBC_1}
Let $X$ and $Y$ be locally compact Hausdorff spaces, and $f\:X\rightarrow Y$ be an FBC map of degree $d$. Then $f$ is open, closed, surjective, and proper. Furthermore, $B_f$ and $V_f$ are nowhere dense.
\end{prop}
\begin{remark}
    Recall that a map $f\:X\rightarrow Y$ is called proper if, for each compact subset $V$ of $Y$, the inverse image $f^{-1}(V)$ is compact.
\end{remark}

\begin{prop}[Ha\"issinsky \& Pilgrim \cite{HP09}]\label{p:FBC_2}
    Suppose that $X$ and $Y$ are connected, locally connected, locally compact Hausdorff spaces. Let $f\:X\rightarrow Y$ be an open, closed, surjective FBC map. Then the following statements hold:
    \begin{enumerate}
        \smallskip
        \item[\textnormal{(i)}]
        If $V\subseteq Y$ is open and connected, and $U\subseteq X$ is a connected component of $f^{-1}(V)$, then $f|_U\:U\rightarrow V$ is also an FBC map.
        
        \smallskip
        \item[\textnormal{(ii)}]
        If $y\in Y$ and $f^{-1}(y)=\{x_1,\,x_2,\,\dots,\,x_k\}$, then there exist arbitrarily small connected open neighborhood $V$ of $y$ and mutually disjoint connected open neighborhoods $U_i$ of $x_i$ for each $i\in\{1,2,\dots,k\}$, such that $f^{-1}(V)=U_1\cup U_2\cup \dots \cup U_k$
        and that $f|_{U_i}\:U_i\rightarrow V$ is an FBC of degree $\deg(f; x_i)$.
        
        \smallskip
        \item[\textnormal{(iii)}]
        If $f(x)=y$, $\{V_n\}_{n\in\N}$ is a sequence of nested open connected sets with $\bigcap_{n\in\N}\,V_n=\{y\}$, and if $\,\widetilde{V}_n$ is the component of $f^{-1}(V_n)$ containing $x$, then $\bigcap_{n\in\N}\,\widetilde{V}_n=\{x\}$.
    \end{enumerate}
\end{prop}

Here we give two technical properties of finite-to-one, open, closed, and surjective continuous maps. We skip the proofs since they are standard.
\begin{prop}\label{p:f to 1,open and close,surjective}
    Let $X$ and $Y$ be topological spaces, and $f\:X\rightarrow Y$ be a finite-to-one, open, closed, and surjective continuous map. If $\,V\subseteq Y$ is open, and $U=f^{-1}(V)$, then $f|_U\:U\rightarrow V$ is finite-to-one, open, closed, and surjective.
\end{prop}
    

\begin{prop}\label{p:component implies onto}
    Let $X$ and $Y$ be locally connected Hausdorff spaces, and $f\:X\rightarrow Y$ a finite-to-one, open, closed, and surjective continuous map. If $V\subseteq Y$ is open and connected, and $U\subseteq X$ is a connected component of $f^{-1}(V)$, then $U$ is open and $f|_U\:U\rightarrow V$ is finite-to-one, open, closed, and surjective.
\end{prop}



\subsection{Topological CXC systems} \label{ss:CXC}

In this subsection, we recall \defn{topological CXC systems}. We refer the reader to \cite[Section~2.2]{HP09} for more details.

\subsubsection{Basic assumptions and the definition}\label{sss:basic assump}
Let $\fX_0$ and $\fX_1$ be topological spaces satisfying the following assumptions:
\begin{enumerate}
    \smallskip
    \item[\textnormal{(i)}] $\fX_0$ and $\fX_1$ are Hausdorff, locally compact, locally connected spaces, each with finitely many connected components.
    \smallskip
    \item[\textnormal{(ii)}]
    $\fX_1$ is a open subset of $\fX_0$ and $\overline{\fX_1}$ is compact in $\fX_0$.
\end{enumerate}

Let $f\:\fX_1\rightarrow\fX_0$ be an FBC map of degree $\deg(f)=d\geq2$. For each $n\in\N_0$ we define 
\begin{equation*}
	\fX_{n+1}\=f^{-1}(\fX_n)
\end{equation*}
and the \defn{repellor} is defined as 
\begin{equation*}
	X\=\{x\in\fX_1\,:\,f^n(x)\in\fX_1 \mbox{ for all } n\in\N_0\,\}.
\end{equation*}
In addition, we make the following technical assumption
\begin{enumerate}
	\smallskip
    \item[\textnormal{(iii)}] $f|_X\:X\rightarrow X$ is an FBC map of degree $d$.
\end{enumerate}
Note that, in this paper, when $f\:\fX_1\to\fX_0$ and the repellor $X$ are clear from the context, we write $f_X\=f|_X$ for convenience.

Referring to \cite[Section~2.2]{HP09}, the following properties hold for $X$ and $\fX_n$, $n\in\N$.
\begin{enumerate}
    \smallskip
    \item[\textnormal{(1)}]
    $f|_{\fX_{n+1}}\:\fX_{n+1}\rightarrow \fX_n$ is an FBC map of degree $d$. 
    \smallskip
    \item[\textnormal{(2)}]
    $\overline{\fX_{n+1}}\subseteq\fX_n$, and  $\overline{\fX_{n+1}}$ is compact in $\fX_n$ since $f$ is proper.
    \smallskip
    \item[\textnormal{(3)}]
    $X$ is totally invariant, i.e., $f^{-1}(X)=X=f(X)$.
    \smallskip
    \item[\textnormal{(4)}]
    For each $k\in\N$, $X=\bigcap\limits_{n\in\N}\overline{\fX_n}
    =\bigcap\limits_{n\in\N}\overline{\fX_{kn}}$, and as a consequence $f$ and $f^k|_{\fX_k}\:\fX_k\rightarrow\fX_0$ have the same repellor $X$.
    \smallskip
    \item[\textnormal{(5)}]
    The definition of the repellor $X$ and the compactness of $\closure{\fX_1}$ implies that given any open set $Y\supseteq X$, $\fX_n\subseteq Y$ for all $n\in\N$ sufficiently large.
\end{enumerate}

\begin{remark}
	In this paper, adopting the terminology on \cite[Page~14]{HP09}, we define a \defn{preimage} under $f$ of a connected set $A$ as a connected component of $f^{-1}(A)$.
\end{remark}

The following is the essential part of the definition of topological CXC systems.

Let $\mathcal{U}_0$ be a finite cover of $X$ by open connected subsets of $\fX_1$. We assume that for each $U\in\cU_0$, $U\cap X\neq\emptyset$. Inductively, we define, for each $n\in\N_0$, 
\begin{equation*}
    \mathcal{U}_{n+1}\=\{\widetilde{U} :\,\mbox{there exists } U\in\cU_n \mbox{ such that } \widetilde{U} \mbox{ is a preimage of } U\}.
\end{equation*}
The elements of $\cU_n$ are connected components of $f^{-n}(U)$, where $U$ ranges over $\cU_0$. We note that, by Proposition~\ref{p:component implies onto}, for each $\tU\in\cU_{n+1}$ and $U\in\cU_n$, if $\tU$ is a preimage of $U$, then $f|_U\:\tU\rightarrow U$ is surjective, and that $f^k(U)\in\cU_{n-k}$ for all $k\in\N_0$ with $k\leq n$. We can see that $\cU_n$ is a finite open cover of $X$ by connected open sets in $\fX_{n+1}$.

We say that $f\:(\fX_1, X)\rightarrow(\fX_0, X)$ is a \defn{topological coarse expanding conformal system} (abbr: topological CXC)  with repellor $X$ if there exists a sequence of finite open covers $\{\mathcal{U}_n\}_{n\in\N_0}$ of $X$ constructed as above, such that the following axioms hold:
\begin{enumerate}
    \smallskip
    \item[\textbf{1.}]\textbf{Expansion Axiom} (abbr: [\textbf{Expans}]): For each finite open cover $\mathcal{V}$ of $X$ by open sets of $\fX_0$, there exists $N\in\N$ such that for each integer $n\geq N$ and each $U\in\mathcal{U}_n$, there exists $V\in\mathcal{V}$ with $U\subseteq V$.
    \smallskip
    \item[\textbf{2.}]\textbf{Irreducibility Axiom} (abbr: [\textbf{Irred}]): For each $x\in X$ and each neighborhood $W$ of $x$ in $\fX_0$, there exists some $n\in\N$ with $f^n(W)\supseteq X$.
    \smallskip
    \item[\textbf{3.}]\textbf{Degree Axiom} (abbr: [\textbf{Deg}]): The set of degrees of maps of the form $f^k|_{\widetilde{U}}\:\widetilde{U}\rightarrow U$, where $U\in\mathcal{U}_n,\ \widetilde{U}\in\mathcal{U}_{n+k}$ and $n,\,k\in\N$ are arbitrary, has a finite maximum denoted by $p$. 
\end{enumerate}

The cover $\cU_n$ will be referred to as the \defn{level-$n$ good open cover} (induced by $\cU_0$). The elements of $\mathcal{U}_n$ will be referred to as \defn{level-$n$ good open sets}.

\subsubsection{Elementary properties} 


\begin{prop}\label{p:same loc deg}
    Assume that $\fX_0$ and $\fX_1$ satisfy the assumptions in Subsection~\ref{sss:basic assump}. Let $f\:\fX_1\rightarrow\fX_0$ be an FBC map with $\deg(f)\geq2$, and $X$ be the repellor. Suppose that $f_X\=f|_X$ is an FBC map such that $\deg(f_X)=\deg(f)$. Then, for each $x\in X$, we have $\deg(f;x)=\deg(f_X;x)$. Furthermore, $\deg(f^n;x)=\deg(f_X^n;x)$ holds for all $x\in X$ and $n\in\N$.
\end{prop}
\begin{proof}
    Fix an arbitrary $x\in X$ and let $y\=f(x)=f_X(x)$. By the definition of FBC maps, we have
    \begin{equation}\label{e:sum of loc deg}
        \sum_{z\in f^{-1}(y)}\deg(f;z)=\deg(f)=\deg(f_X)=\sum_{z\in f_X^{-1}(y)}\deg(f_X;z).
    \end{equation}
    By the total invariance of $X$, we have $f^{-1}(y)=f_X^{-1}(y)$. Since $\deg(f_X;z)\leq\deg(f;z)$ holds for each $z\in X$, we get from (\ref{e:sum of loc deg}) that $\deg(f;z)=\deg(f_X;z)$ holds for each $z\in f^{-1}(y)$. Thus $\deg(f;x)=\deg(f_X;x)$.

    Since local degrees multiply under composition and $\deg(f;z)=\deg(f_X;z)$ holds for each $z\in X$, it follows that $\deg(f^n;x)=\deg(f_X^n;x)$ holds for all $x\in X$ and $n\in\N$.
\end{proof}

The following proposition shows that the repellor $X$ has fractal properties (see \cite[Proposition~2.4.2]{HP09}). 
\begin{prop}[Ha\"issinsky \& Pilgrim \cite{HP09}]\label{p:repellors are fractal}
    Let $f\:(\fX_1,X)\rightarrow(\fX_0,X)$ be a topological CXC system, and $\{\cU_n\}_{n\in\N_0}$ be a sequence of good open covers.
    
    Fix an arbitrary $x\in X$ and a neighborhood $W$ of $x$. Then for each $n\in\N_0$ and each $U\in\cU_n$, there exist $m\in\N$ and $\tU\in\cU_{n+m}$ such that the following statements hold:
    \begin{enumerate}
        \smallskip
        \item[\textnormal{(i)}] $f^m\bigl(\tU\bigr)=U$ and $\closure{\tU}\subseteq W$.
        \smallskip
        \item[\textnormal{(ii)}] $\deg\bigl(f^m|_{\tU}\bigr)\leq \frac{p}{\deg(f^n|_U)}$, where $p$ is the constant from {\rm[{\bf Deg}]}.
    \end{enumerate}
\end{prop}

For a topological CXC system $f\:(\fX_1,X)\rightarrow(\fX_0,X)$, the \defn{post-branch set} is defined to be
\begin{equation*}
    P_f\=X\cap\closure{\bigcup_{n\in\N}V_{f^n}},
\end{equation*}
and we denote
\begin{equation*}
    \post(f)\=X\cap\bigcup_{n\in\N}V_{f^n}.
\end{equation*}
We record the following property of the post-branch set; see \cite[Proposition~2.4.3]{HP09}.
\begin{prop}[Ha\"issinsky \& Pilgrim \cite{HP09}]\label{p:Pf nowhere dense}
    Let $f\:(\fX_1,X)\rightarrow(\fX_0,X)$ be a topological CXC system.
    Then the post-branch set $P_f$ is a nowhere dense subset of $X$.
\end{prop}


\subsection{Metric CXC systems} \label{ss:metric CXC}
\subsubsection{The definition}
With topological CXC systems recalled above, we give a brief introduction to \defn{metric CXC systems}. We refer the reader to \cite[Section~2.5]{HP09} for more details.

First, we recall the notion of \defn{roundness}. Let $Z$ be a metric space and $A$ be a bounded proper subset of $Z$ with a nonempty interior. For each $a\in $ int$(A)$, denote
\begin{equation*}
    \begin{aligned}
        L(A,a)&\=\sup\{ \abs{b-a}:b\in A\}>0,\\
        l(A,a)&\=\sup\{r : r\leq L(A,a) \mbox{ and } B(a,r)\subseteq A\}>0.
    \end{aligned}
\end{equation*}
Then the \defn{roundness} of $A$ is defined as
\begin{equation*}
	\Round(A,a)\=L(A,a)/l(A,a).
\end{equation*}
Note that $\Round(A,a)\in [1,+\infty)$.

For a number $K>0$, we say that $A$ is \defn{$K$-almost round} if $\Round(A,a)\leq K$ for some $a\in A$. If $A$ is $K$-almost round, then there exists a real number $s>0$ such that $B(a,s)\subseteq A \subseteq \closure{B(a, Ks)}$.

Suppose that we have a topological CXC system $f\:\fX_1\rightarrow\fX_0$ with repellor $X$ and a sequence of good open covers $\{\cU_n\}_{n\in\N_0}$. Assume that $\fX_0$ is endowed with a metric compatible with its topology. Note that each $U$ in $\bigcup_{n\in\N_0} \cU_n$ is bounded, since $\overline{\fX_1}$ is compact in $\fX_0$. Then $f$ is called a \defn{metric CXC system} if it satisfies the following two axioms:
\begin{enumerate}
    \smallskip
    \item[\textbf{4.}]\textbf{Roundness Distortion Axiom} (abbr: [\textbf{Round}]): There exist continuous increasing embeddings $\rho_+\:[1,+\infty)\rightarrow [1,+\infty)$ and $\rho_-\:[1,+\infty)\rightarrow [1,+\infty)$ such that 
    for all 
    $n,\,k \in \N_0$, $U\in\cU_n$, $\tU\in\cU_{n+k}$, $y\in U$, and $\ty\in\tU$, if $f^k\bigl(\tU\bigr)=U$ and $f^k(\ty)=y$, then
    \begin{equation*}
    \begin{aligned}
        \Round\bigl(\tU,\ty\bigr)&<\rho_-(\Round(U,y)),\\
        \Round(U,y)&<\rho_+\bigl(\Round\bigl(\tU,\ty\bigr)\bigr).
    \end{aligned}
    \end{equation*}
    
    \smallskip
    \item[\textbf{5.}]\textbf{Diameter Distortion Axiom} (abbr: [\textbf{Diam}]): There exist increasing homeomorphisms 
    $\delta_+\:[0,1]\rightarrow[0,1]$ and $\delta_-\:[0,1]\rightarrow[0,1]$ such that for all $n_0,\,n_1,\,k\in\N_0$, $U\in\cU_{n_0}$, $U'\in\cU_{n_1}$, $\tU\in\cU_{n_0+k}$, and $\tU'\in\cU_{n_1+k}$, if $\tU'\subseteq\tU$, $U'\subseteq U$, $f^k\bigl(\tU\bigr)=U$, and $f^k\bigl(\tU'\bigr)=U'$, then
    \begin{equation*}
    \begin{aligned}
        {\diam\tU'}/{\diam\tU}&<\delta_-({\diam U'}/{\diam U}),\\
        {\diam U'}/{\diam U}&<\delta_+\bigl({\diam \tU'}/{\diam \tU}\bigr).
    \end{aligned}
    \end{equation*}
\end{enumerate}


\subsubsection{Regularity of metric CXC systems} 
In this subsection, we recall some properties of metric CXC systems. The results are recorded below, and more details can be found in \cite[Section~2.6]{HP09}.

The following proposition summarizes the roundness properties given by \cite[Propositions~2.6.2 and 2.6.3]{HP09}.

\begin{prop}[Ha\"issinsky \& Pilgrim \cite{HP09}] \label{p:roundness}
    Let $f\:(\fX_1,X)\rightarrow(\fX_0,X)$ be a metric CXC system, and $\{\cU_n\}_{n\in\N_0}$ be a sequence of good open covers. Then there exists a constant $K>1$ and non-increasing sequences $\{c_n\}_{n\in\N_0}$ and $\{d_n\}_{n\in\N_0}$ of positive real numbers converging to $0$ such that the following statements hold for all $n\in\N_0$ and $x\in X$:
    \begin{enumerate}
    	\smallskip
    	\item[\textnormal{(i)}] $0<c_n\leq\inf \{ \diam U : U\in\cU_n \} \leq \sup \{ \diam U : U\in\cU_n \} \leq d_n$.
    	
        \smallskip
        \item[\textnormal{(ii)}] There exists $U\in\cU_n$ such that U is K-almost round with respect to $x$.
        
        \smallskip
        \item[\textnormal{(iii)}] If $U\in\cU_n$, then there exists $\tx\in X$ such that $\Round(U,\tx) < K$.
        
        \smallskip
        \item[\textnormal{(iv)}] If $0<r<\frac{1}{2K}\min\{\diam U:U\in\cU_n\}$, then there exist $U\in\cU_n$ and $s>r$ such that
        \begin{equation*}
        	B(x,r)\subseteq B(x,s)\subseteq U\subseteq B(x,Ks).
        \end{equation*}
        In particular, let $\delta_n$ be the Lebesgue number of $\cU_n$, then $\delta_n\geq\frac{c_n}{2K}$. 
    \end{enumerate}
\end{prop}

The following proposition summarizes the local comparability and exponential contraction properties given by \cite[Propositions 2.6.4 and 2.6.5]{HP09}.
\begin{prop}[Ha\"issinsky \& Pilgrim \cite{HP09}] \label{p:comparability}
	Let $f\:(\fX_1,X)\rightarrow(\fX_0,X)$ be a metric CXC system, and $\{\cU_n\}_{n\in\N_0}$ be a sequence of good open covers. Then there exist constants $\lambda,\, \theta\in(0,1)$, and $C'>0$ such that for all $n,\,k\in\N_0$, $U'\in\cU_{n+k}$, $U\in\cU_n$, if $U\cap U'\cap X\neq \emptyset$, then the following statements hold:
	\begin{enumerate}
    	\smallskip
        \item[\textnormal{(i)}] $\diam U'/ \diam U \leq C'\theta^k$.
        
        \smallskip
        \item[\textnormal{(ii)}] $\lambda <\diam U'/\diam U < \lambda^{-1}$ if $k=1$.
	\end{enumerate}
	In particular, we may assume that $d_n=C'd_0\theta^n$, where $d_n$ is given in Proposition~\ref{p:roundness}.
\end{prop}

\section{The Assumptions} \label{s:assumption}
We state below the assumptions under which we will develop our theory in most parts of this paper. We will repeatedly refer to such assumptions later.

\begin{assumptions}
\quad
\begin{enumerate}

    \smallskip
    \item[\textnormal{(1)}]  $\fX_0$ and $\fX_1$ are Hausdorff, locally compact, locally connected spaces, each of which has finitely many connected components. $\fX_0$ is endowed with a metric compatible with its topology.
    
    \smallskip
    \item[\textnormal{(2)}] $\fX_1$ is an open subset of $\fX_0$ and $\overline{\fX_1}$ is compact in $\fX_0$.

    \smallskip
    \item[\textnormal{(3)}] $f\:\fX_1\rightarrow\fX_0$ is an FBC map of degree $\deg(f) \eqqcolon d\geq2$ with repellor $X$, and $f|_X\:X\rightarrow X$ is an FBC map of degree $d$.

    \smallskip
    \item[\textnormal{(4)}]  $\cU_0$ is a cover of $X$ by open, connected subsets of $\fX_1$ whose intersection with $X$ is nonempty. $\{\cU_n\}_{n\in\N_0}$ is the sequence of open covers of $X$, where $\cU_n$ is the collection of preimages of elements in $\cU_{n-1}$. Axioms [{\bf Expans}], [{\bf Irred}], [{\bf Deg}], [{\bf Round}], and {\bf [Diam]} hold.
\end{enumerate}
\end{assumptions}

Assumptions (1) through (4) reformulate the definition of metric CXC systems.

We fix some notations for this paper.
In this paper, when $f\:(\fX_1,X)\rightarrow(\fX_0,X)$ and $\{\cU_n\}_{n\in\N_0}$ satisfy the assumptions above, we write $f_X\=f|_X$, and 
\begin{equation}   \label{e:Wn}
	\cW_n\=\{U\cap X : U\in\cU_n\},
\end{equation}
for each $n\in\N_0$.
We note that, in this paper, the terminology \defn{preimage} under $f$ of a connected set $A$ refers to a connected component of $f^{-1}(A)$.

\section{Asymptotic $h$-expansive} \label{s:asymptotic h-expansiveness}
\subsection{Basic concepts}\label{ss:asymptotic h-expans_basic concept}
In this part, we review some concepts from dynamical systems. We refer the reader to \cite[Chapter~3]{PU10}, \cite[Chapter~9]{Wa82} or \cite[Chapter~20]{KH95} for more detailed studies of these concepts.

Let $Z$ be a compact metric space and $g\:Z\rightarrow Z$ a continuous map. 

Let $\xi=\{A_j : j\in J\}$ and $\eta=\{B_k : k\in K\}$ be two covers of $Z$. We say that $\xi$ is a \defn{refinement} of $\eta$ if for each $A_j\in\xi$, there exists $B_k\in\eta$ such that $A_j\subseteq B_k$. For two covers $\xi$ and $\eta$, the \defn{common refinement} $\xi\vee\eta$ of $\xi$ and $\eta$ is defined as
\begin{equation*}
    \xi\vee\eta\=\{A_j\cap B_k : j\in J,\ k\in K\},
\end{equation*}
which is also a cover of $Z$. Note that if $\xi$ and $\eta$ are both open covers, then $\xi\vee\eta$ is also an open cover. Define $g^{-1}(\xi)\=\bigl\{g^{-1}(A_j) : j\in J\bigr\}$, and we denote for each $n\in\mathbb{N}$,
\begin{equation*}
    \xi_g^n \= \bigvee_{j=0}^{n-1}g^{-j}(\xi)=\xi\vee g^{-1}(\xi)\vee\dots\vee g^{-(n-1)}(\xi).
\end{equation*}

\begin{definition}[Refining sequences of open covers]   \label{d:refining_sequence}
    A sequence of open cover $\{\xi_i\}_{i\in\mathbb{N}_0}$ of a compact metric space $Z$ is a \defn{refining sequence of open covers} if the following conditions are satisfied:
    \begin{enumerate}
        \smallskip
        \item[\textnormal{(i)}] $\xi_{i+1}$ is a refinement of $\xi_i$ for each $i\in\mathbb{N}_0$.
        \smallskip
        \item[\textnormal{(ii)}]
        For each open cover $\eta$ of $Z$, there exists $j\in\mathbb{N}$ such that for each integer $i\geq j$, $\xi_i$ is a refinement of $\eta$. 
    \end{enumerate}
\end{definition}

The topological tail entropy was first introduced by M.~Misiurewicz under the name of ``topological conditional entropy" in \cite{Mi73, Mi76}. Here we adopt the terminology in \cite{Do11} (see \cite[Remark~6.3.18]{Do11}).

\begin{definition}[Topological conditional entropy and topological tail entropy]   \label{d:conditional_tail_entropy}
    Let $Z$ be a compact metric space and $g\:Z\rightarrow Z$ be a continuous map. For each pair of open covers $\xi$ and $\eta$ of $Z$, we denote
    \begin{equation}   \label{e:H(A|B)}
        H(\xi|\eta)=\log\Bigl(
        \max\limits_{A\in\eta}\bigl\{
        \min\bigl\{
        \card \xi_A : \xi_A\subseteq\xi,\ A\subseteq\bigcup\xi_A
        \bigr\}
        \bigr\}
        \Bigr)
    \end{equation}
    and
	\begin{equation}\label{e: h(g,A|B)}
    	h(g,\xi|\eta)\=\lim_{n\to +\infty}
    	\frac{1}{n}H\biggl(
    	\bigvee_{i=0}^{n-1} g^{-i}(\xi)
    	\,\bigg|\,
    	\bigvee_{j=0}^{n-1} g^{-j}(\eta)
    	\biggr).
    \end{equation}

    For a given open cover $\eta$, the \defn{topological conditional entropy} $h(g|\eta)$ of $g$ is defined as     
	\begin{equation}\label{e:top cond entropy}
        h(g|\eta)\=
        \lim_{l\to +\infty} h(g,\xi_l|\eta),
    \end{equation}
    where $\{\xi_l\}_{l\in \mathbb{N}_0}$ is an arbitrary refining sequence of open covers.
    
    The \defn{topological tail entropy} $h^{*}(g)$ of $g$ is defined by
	\begin{equation}\label{e:top tail entropy}
        h^*(g)\=
        \lim_{m\to +\infty}
        \lim_{l\to +\infty} 
        h(g,\xi_l|\eta_m),
    \end{equation}
    where  $\{\xi_l\}_{l\in \mathbb{N}_0}$ and  $\{\eta_m\}_{m\in \mathbb{N}_0}$ are two arbitrary refining sequences of open covers.
    
    We note that the limits in (\ref{e:top cond entropy}) and (\ref{e:top tail entropy}) always exist, and both quantities do not depend on the choices of the refining sequences of open covers.
\end{definition}

It should be noted that $h^*$ is well-behaved under iterations, as it satisfies
\begin{equation}    \label{e:tail_entropy_n}
    h^*(g^n)=nh^*(g)
\end{equation}
for each $n\in\mathbb{N}$ and each continuous map $g\:Z\rightarrow Z$ on compact metric space $Z$.

	\begin{definition}[Forward expansive]
	A continuous map $g\:Z\rightarrow Z$ on compact metric space $Z$ is \defn{forward expansive} if there exists $\delta>0$ such that for all distinct points $x,\,y\in Z$, there exists $n\in\N_0$ for which $d(f^n(x),f^n(y))\geq \delta$.
    \end{definition}

The concept of $h$-expansiveness was introduced by R.~Bowen in \cite{Bo72}. We adopt the formulation in \cite{Mi76} (see also \cite{Do11}).

\begin{definition}[$h$-expansive]
    A continuous map $g\:Z\rightarrow Z$ on compact metric space $Z$ is \defn{h-expansive} if there exists a finite open cover $\eta$ of $Z$ such that $h(g|\eta)=0$.
\end{definition}

\begin{rem}\label{r: expans=>h-expans=>asym h-expans}
	It is shown in \cite{Mi76} that a continuous map $g\:Z\to Z$ on compact metric space $(Z,d)$ is $h$-expansive if and only if 
	$$
	\sup_{x\in Z}\biggl\{ \lim_{\delta\to 0}\limsup_{n\to +\infty}\frac{1}{n}\log r_n(g,\Phi_\epsilon(x),\delta)\biggr\}=0
	$$
	for some $\epsilon>0$. Here $\Phi_\epsilon(x)$ is defined as
	$$
	\Phi_\epsilon(x)\=\bigl\{y\in Z:d(g^n(x),g^n(y))\leq\epsilon\mbox{ for each }n\in\N_0\bigr\},
	$$
	and $r_n(g,\Phi_\epsilon(x),\delta)$ stands for the smallest cardinality of a set that $(n,\delta)$-spans $\Phi_\epsilon(x)$. If $g$ is forward expansive, then there is $\epsilon>0$ such that $\Phi_\epsilon(x)=\{x\}$ for each $x\in Z$, which implies that $g$ is $h$-expansive.
\end{rem}

A weaker property called ``asymptotic $h$-expansiveness" was then introduced by M.~Misiurewicz in \cite{Mi73} (see also \cite{Mi76,Do11}).

\begin{definition}[Asymptotic $h$-expansive]    \label{d:asymp_h_exp}
    A continuous map $g\:Z\rightarrow Z$ on a compact metric space $Z$ is \defn{asymptotically h-expansive} if $h^*(g)=0$.
\end{definition}

\begin{remark}
	It is shown in \cite{Mi76} that a continuous map $g\:Z\to Z$ on compact metric space $Z$ is asymptotic $h$-expansive whenever $g$ is $h$-expansive.
\end{remark}


\subsection{Technical lemmas}\label{ss:asym h-expans_lemmas}
In this subsection, we state and prove some technical lemmas.

\begin{lemma} \label{l:exp diam lowerbound}
Assume that $f\:(\fX_1,X)\rightarrow(\fX_0,X)$ and $\{\cU_n\}_{n\in\N_0}$ satisfy the Assumptions in Section~\ref{s:assumption}. Then there exists a constant $C''>0$ such that for each $n\in\N$,
\begin{equation*}
	\inf\limits_{U\in\cU_n}\diam U\geq C''\lambda^n.
\end{equation*}
Here $\lambda \in (0,1)$ is a constant from Proposition~\ref{p:comparability}. In particular, we may let $c_n=C''\lambda^n$ in Proposition~\ref{p:roundness}.
\end{lemma}
\begin{proof}
    Fix arbitrary $n\in\N_0$ and $U'\in\cU_{n+1}$, there exists $U\in\cU_n$ such that $U\cap U'\cap X\neq \emptyset$. Thus by Proposition~\ref{p:comparability}~(i),
    \begin{equation*}
    	\lambda<{\diam U'}/{\diam U} .
    \end{equation*}
    Thus we have 
    \begin{equation*}
    	\diam U'
    > \lambda \diam U
    \geq \lambda  \inf\limits_{V\in\cU_n}\diam V.
    \end{equation*}
    Since $U'\in\cU_{n+1}$ is arbitrary, we have 
    \begin{equation*}
   	\inf\limits_{V'\in\cU_{n+1}}\diam V'
    \geq \lambda \inf\limits_{V\in\cU_n}\diam V.
    \end{equation*}
    Thus, by induction, we get 
    \begin{equation*}
    	\inf\limits_{V\in\cU_n}\diam V
    \geq \lambda^{n}\inf\limits_{V^0\in\cU_0}\diam \bigl(V^0\bigr).
    \end{equation*}
    Define $C'' \=  \lambda^{-1} \cdot\inf_{V^0\in\cU_0}\diam \bigl(V^0\bigr)$, and the proof is complete.
\end{proof}
\begin{remark}
    It should be noted that since $c_n\leq d_n$, we have $0<\lambda\leq\theta<1$, where the constant $\theta\in(0,1)$ is from Proposition~\ref{p:comparability}.
\end{remark}

Now we define the notion of \defn{$(m,k)$-bouquet}, which is the union of some elements in $\cU_m$ and $\cU_{m+k}$, and we show that the diameter of a bouquet can be bounded from above exponentially.
\begin{definition}[Bouquet]   \label{d:bouquet}
    Assume that $f\:(\fX_1,X)\rightarrow(\fX_0,X)$ and $\{\cU_n\}_{n\in\N_0}$ satisfy the Assumptions in Section~\ref{s:assumption}.
    For each pair of $m,k\in\mathbb{N}_0$ and each $U^m\in\cU_m$, we denote
    \begin{equation}
        \Bouquet_m^k(U^m)\=\{U^m\}\cup\bigl\{ U^{m+k}\in\cU_{m+k} : U^{m+k}\cap U^m\cap X\neq\emptyset\bigr\}.
    \end{equation}
    The \defn{$(m,k)$-bouquet} centered at $U^m$ is defined as
    \begin{equation}
        \bouquet_m^k(U^m)
        \=\bigcup\Bouquet_m^k(U^m)
         =U^m\cup\bigcup\bigl\{ U^{m+k}\in\cU_{m+k} : U^{m+k}\cap U^m\cap X\neq\emptyset \bigr\}.
    \end{equation}
\end{definition}

\begin{lemma}[An upper bound for the diameter of bouquet]
\label{l:diam bound bouquet}
    Assume that $f\:(\fX_1,X)\rightarrow(\fX_0,X)$ and $\{\cU_n\}_{n\in\N_0}$ satisfy the Assumptions in Section~\ref{s:assumption}. Then
    for each pair $m,k\in\mathbb{N}_0$, and each $U\in\cU_m$, we have
    \begin{equation}
        \diam \bigl( \bouquet_m^k(U) \bigr) < 2(C'+1)^2d_0\theta^m.
    \end{equation}
    Here the constants $d_0>0$, $C' > 0$, and $\theta\in(0,1)$ are the constants from Propositions~\ref{p:roundness} and~\ref{p:comparability}.
\end{lemma}
\begin{proof}
    Fix arbitrary $m,\,k \in \N_0$, $U\in \cU_m$, and  $a,\,b\in\bouquet_m^k(U)$. Then the following statements hold:
    \begin{enumerate}
        \smallskip
        \item[\textnormal{(i)}] If $a,b\in U$, then $\abs{a-b}\leq\diam U$.
        
        \smallskip
        \item[\textnormal{(ii)}] If there exists $U'\in\cU_{m+k}$ with $U\cap U'\cap X\neq\emptyset$ such that $a\in U$, $b\in U'$, then by Proposition~\ref{p:comparability}~(i),
        \begin{equation*}
            {\diam U'}/{\diam U}\leq C'\theta^k,
        \end{equation*}
        and consequently $\abs{a-b}\leq \bigl( C'\theta^k+1 \bigr)\diam U$.

        \smallskip
        \item[\textnormal{(iii)}] 
        If there exist $U',\, U''\in\cU_{m+k}$ with $U\cap U'\cap X\neq\emptyset$ and $U\cap U''\cap X\neq\emptyset$, such that $a\in U'$, $b\in U''$, 
        then by a similar argument as in (ii), we get $\abs{a-b}\leq 2 \bigl( C'\theta^k+1 \bigr) \diam U$.
    \end{enumerate}
    Therefore, $\diam \bigl( \bouquet_m^k(U) \bigr)
      \leq 2\bigl( C'\theta^k+1 \bigr) \diam U
       <     2(C'+1)^2d_0\theta^m$.
\end{proof}

In what follows, we show some properties of subsets of the form 
\begin{equation}
    A=\bigcap_{i=0}^n f_X^{-i}(W_i^m)
    =\left\{ x\in W_0^m : f_X^i(x)\in W^m_i,\ i\in\{1,2,\dots,n\} \right\},
\end{equation}
where $W_i^m\in\cW_m$ for each $i\in \{0,1,\dots n\}$.
The lemma below is a straightforward consequence of the definitions.

\begin{lemma}[Inverse images of $W^n$]\label{l:inv image of W^n}
    Assume that $f\:(\fX_1,X)\rightarrow(\fX_0,X)$ and $\{\cU_n\}_{n\in\N_0}$ satisfy the Assumptions in Section~\ref{s:assumption}. 
    Then for each $n\in\N_0$ and each $W^n\in \cW_n$, if $W^n=U^n\cap X$ for some $U^n\in\cU_n$, then
    \begin{equation}
        f_X^{-1}(W^n)=f^{-1}(U^n)\cap X
        =X\cap\bigcup\bigl\{ U^{n+1} \in\cU_{n+1} : f\bigl(U^{n+1}\bigr) = U^{n} \bigr\}.
    \end{equation}
\end{lemma}

We construct a class of open covers for subsets of the form $A=\bigcap_{i=0}^n f_X^{-i}(W_i^m)$.
For all $m \in\N_0$, $n\in\N$, and for an arbitrary choice of $U^m_i\in\cU_m$ for each $i\in\{0,1,\dots,n\}$, we denote
\begin{equation}   \label{e:Em}
  \begin{aligned}
        &E_m(U^m_0,\dots,U^m_{n-1};U^m_n) \\
        &\qquad \= \bigl\{U^{m+n}\in\cU_{m+n} : f^n(U^{m+n})=U^m_n  \text{ and } f^i(U^{m+n})\cap U^m_i\cap X\neq\emptyset \\
        &\qquad\qquad\qquad\qquad\qquad\qquad\qquad\qquad\qquad\qquad \text{ for each } i\in\{0,\dots, n-1\}\bigr\}.
  \end{aligned}  
\end{equation}

\begin{lemma}\label{l:cover by E_m}
    Assume that $f\:(\fX_1,X)\rightarrow(\fX_0,X)$ and $\{\cU_n\}_{n\in\N_0}$ satisfy the Assumptions in Section~\ref{s:assumption}. 
    Then for all $m\in\mathbb{N}_0$, $n\in\N$, and an arbitrary choice of $U^m_i\in\cU_m$ for each $i\in\{0,1,\dots,n\}$, we have
    \begin{equation}
        \bigcap_{i=0}^n f_X^{-i}(U_i^m\cap X)
        \subseteq
        \bigcup_{U^{m+n}\in E_m(U^m_0,\dots,U^m_{n-1};U^m_n)} U^{m+n}\cap X.
    \end{equation}
\end{lemma}
\begin{proof}
    For each $x\in\bigcap_{i=0}^n f_X^{-i}(U_i^m\cap X)=\bigl\{ z\in U_0^m\cap X : f_X^i(z)\in U^m_i\cap X$ for each $i\in\{1,2,\dots,n\} \bigr\}$, there exists $U^{m+n}\in\cU_{m+n}$ such that $x\in U^{m+n}$ and $f^n(U^{m+n})=U^m_n$. It follows immediately that $f^i(x)\in f^i(U^{m+n})\cap U^m_i\cap X$ and thus $f^i(U^{m+n})\cap U^m_i\cap X\neq\emptyset$ for each $i\in\{1,\,2,\,\cdots,\,n\}$. Therefore $U^{m+n} \in E_m(U^m_0,\dots,U^m_{n-1};U^m_n)$ and  the proof is now complete.
\end{proof}

The lemma below gives an upper bound for the number of elements in $\cW_{n+k}$ needed to cover an element in $\cW_n$.
\begin{lemma} \label{l:covered by smaller}
    Assume that $f\:(\fX_1,X)\rightarrow(\fX_0,X)$ and $\{\cU_n\}_{n\in\N_0}$ satisfy the Assumptions in Section~\ref{s:assumption}. 
    Then there exists a constant $T_0\geq 1$ such that for all $n\in\N_0$, $k\in\N$, and $W^n\in\cW_n$, there exists $\cI\subseteq\cW_{n+k}$ with $\card(\cI)\leq (pT_0)^k$ such that 
    \begin{equation*}
    W^n\subseteq \bigcup\cI,
    \end{equation*}
    where $p$ is the constant from {\rm[{\bf Deg}]}.
\end{lemma}
\begin{proof}
    By Lemma~\ref{l:diam bound bouquet} there exists $N\in\N$ such that for each $U^{N}\in\cU_{N}$, there exists $U^0\in\cU_0$ such that $\bouquet_{N}^1 \bigl( U^{N} \bigr) \subseteq U^0$. 

    Fix arbitrary integer $n\geq N$ and $U^n\in\cU_n$, and consider $\bouquet_n^1(U^n)$. We have 
    \begin{equation*}
    	f^{n-N}  \bigl( \bouquet_n^1(U^n) \bigr) 
    	\subseteq \bouquet_N^1 \bigl( f^{n-N}(U^n) \bigr)
    	\subseteq U^0
    \end{equation*}
    for some $U^0\in\cU_0$. Since $U^n$ is a preimage under $f^{n-N}$ of $f^{n-N}(U^n)$, there exists a preimage under $f^{n-N}$ of $U^0$, denoted by $U^{n-N}\in\cU_{n-N}$, such that $U^n\subseteq U^{n-N}$. For each $U^{n+1}\in\Bouquet_n^1(U^n) \cap \cU_{n+1}$, there exists $U'\in\cU_{n-N}$ which is a preimage of $U^0$ such that $U^{n+1}\subseteq U'$. Since $U^{n+1}\cap U^n\neq \emptyset$ and different preimages of $U^0$ are disjoint, we have $U'=U^{n-N}$, thus $\bouquet_n^1(U^n)\subseteq U^{n-N}$.

    Consider $f^{n-N}|_{U^{n-N}}\:U^{n-N}\rightarrow U^0$. By [{\bf Deg}], this map is at most $p$ to $1$, thus we have
    \begin{equation*}
        \card\left(\Bouquet_n^1(U^n)\right)
        \leq p\cdot\card\bigl( \Bouquet_N^1\bigl(f^{n-N}(U^n) \bigr) \bigr).
    \end{equation*}
    Let $T_0\=\max\limits_{0\leq i\leq N}\sup \bigl\{ \card \bigl(\Bouquet_i^1(U )\bigr)  :   U \in \cU_i   \bigr\}$,
    then we have
    \begin{equation*}
    \card\left(\Bouquet_n^1(U^n)\right)\leq pT_0 .
    \end{equation*}
    
    Define $W^n \= U^n\cap X$. Then we have
    \begin{equation*}
    \begin{aligned}
        W^n\subseteq &X\cap\bigcup \left\{ 
        U^{n+1}\in\cU_{n+1} : U^{n+1}\in\Bouquet_n^1(U^n) 
        \right\}\\
        =&\left\{ 
        X\cap U^{n+1}\in\cW_{n+1} : U^{n+1}\in\Bouquet_n^1(U^n) 
        \right\}.
    \end{aligned}
    \end{equation*}
    Thus each element of $\cW_n$ can be covered by at most $p T_0$ many elements in $\cW_{n+1}$. Inductively, for every $k\in\N$, each element in $\cW_n$ can be covered by $(p T_0)^k$ many elements in $\cW_{n+k}$. We have established the lemma under the assumption that $n\geq N$.

    Consider now $n\in\{0,\,\dots,\,N\}$. By the definition of $T_0$, each element in $\cW_n$ can be covered by $T_0$ many elements in $\cW_{n+1}$. Thus, $T_0$ is the constant as desired.
\end{proof}

\subsection{Proofs of Theorems~\ref{t:Intro asym h-expans} and \ref{t:Intro_expansive and h-expansive}}\label{ss:pf thm asym & not h-expans}
In this subsection, we will establish the asymptotic $h$-expansiveness of the map $f|_X\:X\rightarrow X$. Our strategy is to prove the result under slightly stronger assumptions (Theorem~\ref{t:asymptotic h-expans}), then we show that for some $N\in\N$, $(f|_X)^N$ satisfies the stronger assumptions, resulting in $h^*\left((f|_X)^N\right)=Nh^*(f|_X)=0$.

Here is the main idea of the proof of Theorem~\ref{t:asymptotic h-expans}. Recall that we construct in Lemma~\ref{l:cover by E_m} a cover $E_m(U_0^m,\cdots,U_{n-1}^m;U_n^m)$ for each set of the form $\bigcap_{i=0}^n f^{-i}{(U_i^m\cap X)}$ by level-$(m+n)$ good open sets, where $m\in\N_0$, $n\in\N$, and each $U_i^m\in\cU_m$ is a level-$m$ good open set. Then, we give an upper bound for $\card(E_m(U_0^m,\cdots,U_{n-1}^m;U_n^m))$ that depends only on $m,\,n$.  Combining Lemma~\ref{l:covered by smaller}, we have that for an arbitrary $A\in\bigvee_{i=0}^{n-1}f^{-i}(\cW_m)$, the number of sets in the open cover $\bigvee_{j=0}^{n-1}f^{-j}(\cW_l)$ needed to cover $A$ can be bounded from above, and such an upper bound leads to the conclusion that the topological tail entropy $h^*(f)=0$. Bouquets are crucially used in the proof.

First, we note that $X$ is indeed a compact metric space. By $X=\bigcap_{n\in\N_0}\closure{\fX_n}$, $X$ is closed in $\closure{\fX_1}$. Since $\closure{\fX_1}$ is compact, $X$ is compact in $\closure{\fX_1}$, thus compact with respect to its own subspace topology.

\begin{lemma}\label{l:Wn_refining}
	Let $f\:(\fX_1,X)\rightarrow(\fX_0,X)$ and $\{\cU_n\}_{n\in\N_0}$ satisfy the Assumptions in Section~\ref{s:assumption}. Assume that for each $n\in\N_0$, $\cU_{n+1}$ is a refinement of $\cU_n$. Then $\{\cW_m\}_{m\in\N_0}$ is a refining sequence of open covers of $X$.
\end{lemma}


\begin{proof}
 		By the hypothesis on $\cU_n$, $\cW_{n+1}$ is a refinement of $\cW_n$ for each $n\in\N_0$. For each open cover $\cY$ of $X$, set $\cY=\{V\cap X : V\in\cV\}$ where $\cV$ is a cover of $X$ by open subsets of $\fX_1$. Since $X$ is compact, we may assume $\cV$ to be a finite cover. Then by [{\bf Expans}], when $n\in\N$ is sufficiently large, for each $U\in\cU_n$ there exists $V\in\cV$ such that $U\subseteq V$. Thus for each $W\in\cW_n$, there is $Y\in\cY$ such that $W\subseteq Y$. So $\cW_n$ is a refinement of $\cY$. By Definition~\ref{d:refining_sequence}, therefore, $\{\cW_m\}_{m\in\N_0}$ is a refining sequence of open cover of $X$.
\end{proof}

\begin{theorem}\label{t:asymptotic h-expans}
Let $f\:(\fX_1,X)\rightarrow(\fX_0,X)$ and $\{\cU_n\}_{n\in\N_0}$ satisfy the Assumptions in Section~\ref{s:assumption}. 
Assume that for each $n\in\N_0$, $\cU_{n+1}$ is a refinement of $\cU_n$. Then $f|_X\:X\rightarrow X$ is asymptotically h-expansive.
\end{theorem}
\begin{proof}
 By Lemma~\ref{l:Wn_refining}, $\{\cW_m\}_{m\in\N_0}$ (defined in (\ref{e:Wn})) is a refining sequence of open covers of $X$.
 
 Fix arbitrary $m,\,l,\,n\in\N$, and $A\in\bigvee_{i=0}^nf_X^{-i}(\cW_m)$ with $m<l<n$. We construct an upper bound for the set
\begin{equation*}
        \min\biggl\{
        \card \xi_A : \xi_A\subseteq\bigvee_{j=0}^nf_X^{-j}(\cW_l),\ A\subseteq\bigcup\xi_A
        \biggr\},
\end{equation*}

\smallskip
\textit{Step 1}. Cover elements in $\bigvee_{i=0}^nf_X^{-i}(\cW_m)$ with elements in $\cW_{m+n}$.

Choose $W_i^m\in\cW_m$ for each $i\in \{0,1,\dots n\}$ such that
\begin{equation*}
    A=\bigcap_{i=0}^n  f_X^{-i}(W_i^m)
    =\left\{ x\in W_0^m : f_X^i(x)\in W^m_i,\ i\in\{1,2,\dots,n\} \right\}.
\end{equation*}
Then for each $i\in\{0,1,\dots,n\}$, $W^m_i=U^m_i\cap X$ for some $U^m_i\in\cU_m$.

By Lemma~\ref{l:cover by E_m}, we get
 \begin{equation}
        A
        \subseteq
        \bigcup_{U^{m+n}\in E_m(U^m_0,\dots,U^m_{n-1};U^m_n)} U^{m+n}\cap X.
    \end{equation}
Then we give an upper bound of the cardinality of the set $E_m(U^m_0,\dots,U^m_{n-1}; U^m_n)$. For each $k \in \{0, \, 1,\, \dots,\, n\}$ let
\begin{equation}\label{e: set: L_k}
    \cL_k\=\bigl\{ f^k(U) : U\in E_m(U^m_0,\dots,U^m_{n-1};U^m_n) \bigr\}.
\end{equation}
By the definition of the set $E_m(U^m_0,\dots,U^m_{n-1};U^m_n)$ in (\ref{e:Em}), we know that $\cL_k\subseteq\cU_{m+n-k}$, and in particular, $\cL_0=E_m(U^m_0,\dots,U^m_{n-1};U^m_n)$ and $\cL_n=\{U^m_n\}$.
By Definition~\ref{d:bouquet} and the definition of the set $E_m(U^m_0,\dots,U^m_{n-1};U^m_n)$ in (\ref{e:Em}), we have
\begin{equation*}
    \bigcup \cL_k\subseteq\bouquet_m^{n-k}(U^m_{k})
\end{equation*}
for each $k\in\{0,1,\dots,n-1\}$.

Put
\begin{equation}\label{e: constant: M_m}
    M_m\=
    \left\lfloor
    m\log_{\lambda} \theta 
    +\log_{\lambda}\left( \frac{4K(C'+1)^2d_0}{C''} \right)
    -1
    \right\rfloor,
\end{equation}
where the constants $K>1$ and $d_0>0$ are from Proposition~\ref{p:roundness}, $\lambda,\, \theta\in(0,1)$, and $C'>0$ are from Proposition~\ref{p:comparability},  and $C''>0$ is from Proposition~\ref{l:exp diam lowerbound}.

We can check that $\{M_m\}_{m\in\N_0}$ is increasing to $+\infty$
and that
\begin{equation}
    2(C'+1)^2d_0\theta^m < {C''\lambda^{M_m}}/{(2K)}.
\end{equation}
By Lemma~\ref{l:diam bound bouquet}, Proposition~\ref{p:roundness}~(iv), and $c_{M_n} = C'' \lambda^{M_n}$ (see Lemma~\ref{l:exp diam lowerbound}), for each $k\in\{0,1,\dots n\}$ we have
\begin{equation*}
    \diam \bouquet_m^{n-k}(U^m_k)< C'' \lambda^{M_n} / (2K)\leq \delta_{M_m},
\end{equation*}
where $\delta_{M_m}$ is the Lebesgue number of $\cU_{M_m}$, which is from Proposition~\ref{p:roundness}~(iv). Thus, there exists $\hU_k\in\cU_{M_m}$ such that
\begin{equation*}
    \bigcup \cL_k\subseteq \bouquet_m^{n-k}(U^m_{k})\subseteq\hU_k.
\end{equation*}

For each $0\leq k\leq n$ and each $0\leq t\leq \min\{n-k,M_m\}$, consider $f^t|_{\hU_k}\:\hU_k\rightarrow f^t(\hU_k)$. By [{\bf Deg}], the map is at most $p$ to 1, thus $\card(\cL_k)\leq p\cdot\card(\cL_{k+t})$. Indeed, if we assume, on the contrary, that $\card(\cL_k)>p\cdot\card(\cL_{k+t})$, then there exists at least one  collection $\bigl\{\tU_1,\,\dots,\tU_{p+1}\bigr\}\subseteq\cL_k$ such that $f^t\bigl(\tU_1\bigr)=\cdots=f^t\bigl(\tU_{p+1}\bigr)=\tU\in\cL_{k+t}$. Since each $\tU_i$, $1\leq i\leq p+1$, is a preimage of $\tU$ (under $f^{-t}$), $\tU_1,\,\dots,\tU_{p+1}$ are pairwise disjoint. Hence, given a point $y\in\tU$, by choosing a preimage $x_i\in\tU_i$ for each $0\leq i\leq p+1$, we obtain a collection of $p+1$ preimages of $y$ in $\hU_k$, which yields a contradiction.

Suppose that $n=s\cdot M_m+t$ for some $s,\,t\in\N_0$ with $0\leq t<M_m$. Then, we have 
$$
\card(\cL_0)\leq p\cdot\card(\cL_{M_m})\leq\cdots\leq p^s\cdot\card(\cL_{s\cdot M_m})\leq p^{s+1}\cdot\card(\cL_n),
$$
which implies
\begin{equation}
    \card (E_m(U^m_0,\dots,U^m_{n-1};U^m_n))\leq p^{\frac{n}{M_m}+1}.
\end{equation}
Define $\cI_A\=\{U\cap X:U\in E_m(U^m_0,\dots,U^m_{n-1};U^m_n)\}\subseteq\cW_{m+n}$. Then $\card(\cI_A)\leq p^{\frac{n}{M_m}+1}$, and $A\subseteq\bigcup\cI_A$.

\smallskip

\textit{Step 2}. Cover elements in $\cI_A$ with elements in $\cW_{l+n}$.

By Lemma~\ref{l:covered by smaller}, for each $W^{m+n}\in\cI_A$ there exists $\cI_{W^{m+n}}\subseteq\cW_{l+n}$ with $\card(\cI_{W^{m+n}})\leq (pT_0)^{l-m}$ such that $W^{m+n}\subseteq \bigcup\cI_{W^{m+n}}$,
 where $p$ is the constant in [{\bf Deg}] and $T_0$ is from Lemma~\ref{l:covered by smaller}.

\smallskip

\textit{Step 3}. Each element in $\cW_{l+n}$ can be covered by some element in $\bigvee_{j=0}^n f_X^{-j}(\cW_l)$.

For each $W^{l+n}\in\cW_{l+n}$, suppose that $W^{l+n}=U^{l+n}\cap X$, $U^{l+n}\in\cU_{l+n}$. By the hypothesis in the statement of Theorem~\ref{t:asymptotic h-expans} that $\{\cU_n\}_{n\in\N_0}$ is refining, for each $j\in\{0,1,\dots,n\}$, there exists $U^{l+j}\in\cU_{l+j}$ such that $U^{l+n}\subseteq U^{l+j}$. Since $U^{l+j}\subseteq f^{-j} \bigl(U^l_j \bigr)$ for some $U^l_j\in\cU_l$, we have $U^{l+n}\subseteq U^{l+j}\subseteq f^{-j}\bigl(U^l_j\bigr)$. As a consequence
\begin{equation}
\begin{aligned}
    W^{l+n}=X\cap U^{l+n}
    \subseteq& X\cap\bigcap_{j=0}^n f^{-j}\bigl(U^l_j\bigr)\\
    =&\bigcap_{j=0}^n \bigl(X\cap f^{-j}\bigl(U^l_j\bigr)\bigr)
    =\bigcap_{j=0}^n f^{-j}\bigl(X\cap U^l_j\bigr)\in\bigvee_{j=0}^nf^{-j}(\cW_l).
\end{aligned}
\end{equation}

\smallskip

By Steps 1 through 3, there exists $\cJ_A\subseteq\bigvee_{j=0}^nf_X^{-j}(\cW_l)$ with $\card(\cJ_A)\leq p^{\frac{n}{M_m}+1}(pT_0)^{l-m}$, such that $A\subseteq\bigcup\cJ_A$. It follows that
\begin{equation*}
        \min\biggl\{
        \card \xi_A : \xi_A\subseteq\bigvee_{j=0}^nf_X^{-j}(\cW_l),\ A\subseteq\bigcup\xi_A
        \biggr\}
        \leq
        p^{\frac{n}{M_m}+1}(pT_0)^{l-m}.
\end{equation*}
Since $A$ is an arbitrary element of $\bigvee_{i=0}^nf_X^{-i}(\cW_m)$, by (\ref{e:H(A|B)}) in Definition~\ref{d:conditional_tail_entropy} we have
\begin{equation*}
\begin{aligned}
    &H\biggl(
    \bigvee_{j=0}^{n} f_X^{-j}(\cW_l)
    \,\bigg|\,
    \bigvee_{i=0}^{n} f_X^{-i}(\cW_m)
    \biggr)\\
    &\qquad=
    \log\biggl(
    \max\limits_{A\in\bigvee_{i=0}^nf_X^{-i}(\cW_m)}\biggl\{
    \min\biggl\{
    \card \xi_A : \xi_A\subseteq\bigvee_{j=0}^nf_X^{-j}(\cW_l),\ A\subseteq\bigcup\xi_A
    \biggr\}
    \biggr\}
    \biggr)\\
    &\qquad\leq
    \log\left( p^{\frac{n}{M_m}+1}(pT_0)^{l-m}\right).
\end{aligned}
\end{equation*}
Thus by (\ref{e:top tail entropy}) in Definition~\ref{d:conditional_tail_entropy}, we have
\begin{equation}   \label{e:Pf_t_asymptotic h-expans_tail_entropy_0}
    \begin{aligned}
        h^*(f_X)
        &=
        \lim_{m\to +\infty}
        \lim_{l\to +\infty} 
        \lim_{n\to +\infty}
        \frac{1}{n}H\biggl(
        \bigvee_{j=0}^{n-1} f_X^{-j}(\cW_l)
        \,\bigg|\,
        \bigvee_{i=0}^{n-1} f_X^{-i}(\cW_m)
        \biggr)\\
        &\leq
        \lim_{m\to +\infty}
        \lim_{l\to +\infty} 
        \lim_{n\to +\infty}
        \frac{1}{n}\log\left(
        p^{\frac{n-1}{M_m}+1}(pT_0)^{l-m}
        \right)
        =
        \lim_{m\to +\infty}\frac{\log p}{M_m}
        =
        0.
    \end{aligned}
\end{equation}
The theorem follows, therefore, from (\ref{e:Pf_t_asymptotic h-expans_tail_entropy_0}) and Definition~\ref{d:asymp_h_exp}.
\end{proof}

\begin{lemma}    \label{l:iteration_satisfy_assumptions}
	Let $f\:(\fX_1,X)\rightarrow(\fX_0,X)$ and $\{\cU_k\}_{k\in\N_0}$ satisfy the Assumptions in Section~\ref{s:assumption}. Assume $n\in\N$. Then $f^n|_{\fX_n}\:(\fX_n,X)\rightarrow(\fX_0,X)$ and $\{\cU_{kn}\}_{k\in\N_0}$ satisfy the Assumptions in Section~\ref{s:assumption}. Moreover, there exists $N\in\N$ such that $\cU_{(k+1)N}$ is a refinement of  $\cU_{kN}$ for each $k\in\N_0$.
\end{lemma}

\begin{proof}
	We first verify the last statement. By [{\bf Expans}], there exists $N\in\N$ such that for each $V'\in\cU_N$ there exists $V\in\cU_0$ such that $V'\subseteq V$. Then for each $k\in\N_0$ and each $U'\in\cU_{(k+1)N}$, there exists $U_0\in\cU_0$ such that $f^{kN}(U')\subseteq U_0$. Since $U'$ is connected, there exists $U\in\cU_{kN}$ such that $U$ is a connected component of $f^{-kN}(U_0)$ and $U'\subseteq U$. Thus $\cU_{(k+1)N}$ is a refinement of  $\cU_{kN}$.
	
	Next, we check the Assumptions in Section~\ref{s:assumption} for $f^n|_{\fX_n}\:(\fX_n,X)\rightarrow(\fX_0,X)$ and $\{\cU_{kn}\}_{k\in\N_0}$ one by one.
	
	\smallskip
	
	(1)	$\fX_n$ is Hausdorff since $\fX_0$ is Hausdorff. 
		
	For each $x\in\fX_n$, choose a compact neighborhood $U_x$ of $f^n(x)$ in $\fX_0$. Since $f$ is proper, $f^{-n}(U_x)$ is a compact neighborhood of $x$ in $\fX_n$, so $\fX_n$ is locally compact. 
	
	Since $\fX_n=f^{-n}(\fX_0)$ is open, for each $x\in\fX_n$ and each neighborhood $U$ of $x$ in $\fX_n$, $U$ is an open subset of $\fX_0$. By the local connectivity of $\fX_0$, there exists a connected neighborhood $V$ of $x$ such that $V\subseteq U\subseteq\fX_n$, thus $\fX_n$ is locally connected. 
	
	Let $\fC_n$ be an arbitrary connected component in $\fX_n$ and $\fC_0$ be the connected component in $\fX_0$ which contains $f^n(\fC_n)$. Since $\fX_0$ is locally connected, $\fC_0$ is open. Then $\fC_n$ is a connected component of $f^{-n}(\fC_0)$. By Proposition~\ref{p:f to 1,open and close,surjective}, $f^n|_{\fX_n}\:{\fX_n\rightarrow\fX_0}$ is finite-to-one, open, closed, and surjective. Then by Proposition~\ref{p:component implies onto}, ${f^n|_{\fC_n}\:\fC_n\rightarrow\fC_0}$ is surjective. Since $\fX_0$ has finitely many connected components and $\deg(f^n)< + \infty$, we get that $\fX_n$ has finitely many connected components.
	
	\smallskip
	(2) Since $f$ is continuous, $\fX_n=f^{-(n-1)}(\fX_1)$ is open. Since $\closure{\fX_n}$ is compact in $\fX_{n-1}$, $\closure{\fX_n}$ is compact in $\fX_0$.
	
	\smallskip
	(3) Let $\deg(f)=d>1$. The composition of two FBC maps is an FBC map, and the degree multiplies, so $\deg(f^n)=d^n$.
	
	Since for each $k\in\N$, $f|_{\fX_k}\:\fX_k\rightarrow\fX_{k-1}$ is an FBC map of degree $d$, $f^n|_{\fX_n}$ is an FBC map of degree $d^n$.  The repellor of $f^n|_{\fX_n}\:\fX_n\rightarrow\fX_0$ is $X=\bigcap_{k\in\N}\overline{\fX_k}
	=\bigcap_{k\in\N}\overline{\fX_{kn}}$.
	Since $f|_X\:X\rightarrow X$ is an FBC map of degree $d$, $f^n|_X=(f|_X)^n$ is an FBC map of degree $d^n$.
	
	\smallskip
	(4) By definition $\cU_n$ is a finite cover of $X$ by connected open subsets of $\fX_n$ and for each $k\in\N_0$ the elements of $\cU_{(k+1)n}$ are preimages under $f^n$ of elements in $\cU_{kn}$. Axioms [{\bf Expans}], [\textbf{Irred}], [\textbf{Deg}], [\textbf{Round}], and [{\bf Diam}] hold for the sequence $\{\cU_{kn}\}_{k\in\N}$, which is a straightforward consequence of the definitions.
\end{proof}

Finally, we give a proof of Theorem~\ref{t:Intro asym h-expans}.

\begin{proof}[Proof of Theorem~\ref{t:Intro asym h-expans}]
    Let $f\:(\fX_1,X)\rightarrow(\fX_0,X)$ and $\{\cU_n\}_{n\in\N_0}$ satisfy the Assumptions in Section~\ref{s:assumption}.

    Since $h^*((f|_X)^n)=nh^*(f|_X)$ for each $n\in\N$ (see (\ref{e:tail_entropy_n})), by Definition~\ref{d:asymp_h_exp} it suffice to prove $h^*\bigl((f|_X)^{N}\bigr)=0$ for some $N\in\N$. 

    By Lemma~\ref{l:iteration_satisfy_assumptions}, there exists $N\in \N$ such that $f^N|_{\fX_N}\:(\fX_N,X)\rightarrow(\fX_0,X)$ and $\{\cU_{kN}\}_{k\in\N_0}$ satisfy the Assumptions in Section~\ref{s:assumption}, and moreover, $\{ \cU_{kN} \}_{k\in\N}$ satisfies the property that for each $U'\in\cU_{(k+1)N}$ there exists $U\in\cU_{kN}$ such that $U'\subseteq U$.
    
    Then by Theorem~\ref{t:asymptotic h-expans}, $(f|_X)^N=f^N|_X$ is asymptotically $h$-expansive. By Definition~\ref{d:asymp_h_exp}, $h^* \bigl( (f|_X)^N \bigr)=0$. Therefore $h^*(f|_X)=0$ by (\ref{e:tail_entropy_n}), and the theorem follows from Definition~\ref{d:asymp_h_exp}.
\end{proof}


    In what follows, we show that a metric CXC system is forward expansive if there is no branch point in the repellor. The proof is based on the observation that the distance between two points $x,\,y\in X$ can be estimated by the diameter of some good open set containing one of $x,\,y$, and such estimation works well under iterations when $d(x,y)$ is small.
    \begin{theorem}\label{t:expansive}
    	Let $f\:(\fX_1,X)\rightarrow(\fX_0,X)$ and $\{\cU_n\}_{n\in\N_0}$ satisfy the Assumptions in Section~\ref{s:assumption}. 
    	Assume that $X\cap B_f=\emptyset$. Then $f|_X$ is forward expansive.
    \end{theorem}
    \begin{proof}
    	For each $x\in X$, since $\deg(f;x)=1$, there exists an open neighborhood $U_x$ of $x$, such that $f|_{U_x}$ is injective. Let $\cV\=\{U_x:x\in X\}$. By Lemma~\ref{l:diam bound bouquet}, we may fix $n_0\in\N$, such that for all $n,\,k\in\N_0$ with $n\geq n_0$ and each $U^n\in\cU_n$, there exists $V\in\cV$ for which $\bouquet_n^k(U^n)\subseteq V$. Thus, $f^{n-n_0}|_{\bouquet_n^k(U^n)}$ is injective for all $n,\,k\in\N_0$ with $n\geq n_0$ and each $U^n\in\cU_n$.
    	
    	Consider 
    	$$\delta\=\min\biggl\{\frac{C''\lambda^{n_0}}{2K},\, \frac{C''\lambda^{n_0}}{2\rho_+(K)}\biggr\},$$
    	where $C''>0$, $\lambda\in(0,1)$ are the constants in Lemma~\ref{l:exp diam lowerbound}, $K>1$ is the constant in Proposition~\ref{p:roundness}, and $\rho_+$ is the increasing embedding in [{\bf Round}].
    	
    	Fix arbitrary $x,\,y\in X$, $x\neq y$. Suppose that 
    	$d(x,y)<\delta\leq\frac{C''\lambda^{n_0}}{2K}$. 
    	Then by Proposition~\ref{p:roundness}~(iv), there exists $U^{n_0}\in\cU_{n_0}$ containing $x,\,y$. Put
    	$$
    	n\=\max\{l\in\N:\mbox{there exists }U\in\cU_l\mbox{ such that }x,\,y\in U\},
    	$$
    	then we have $n\geq n_0$.
    	
    	Suppose that $U\in\cU_n$ and $x,\,y\in U$. By Proposition~\ref{p:roundness} (ii), there exists $U'\in\cU_{n+1}$ such that $x\in U$, and that $\Round(U',x)<K$. Here $K$ is the constant in Proposition~\ref{p:roundness}. By the choice of $n$, we have that $y\notin U'$.
    	
    	Denote $k\=n-n_0$. Since $U\cup U'\subseteq \bouquet_n^1(U)$, the map $f^{k}|_{U\cap U'}$ is injective. Thus, $f^{k}(y)\notin f^{k}(U')$.
    	
    	By [{\bf Round}], we have
    	$$
    	\Round\bigl(f^{k}(U'),f^{k}(x)\bigr)
    	<\rho_+(\Round(U',x))<\rho_+(K).
    	$$
    	Denote $L\=\Round\bigl(f^l(U'),f^k(x)\bigr)$. Then, there exists $r>0$, such that 
    	$$
    	B\bigl(f^k(x),r\bigr)\subseteq f^k(U')\subseteq \closure{B\bigl(f^k(x),Lr\bigr)}.
    	$$
    	It is clear that $$r\geq\frac{\diam(f^k(U'))}{2L}\geq\frac{\diam(f^k(U'))}{2\rho_+(K)}.$$
    	Since $f^k(y)\notin f^k(U')$, we have that 
    	$$
    	d\bigl(f^k(x),f^k(y)\bigr)\geq r\geq\frac{\diam(f^k(U'))}{2\rho_+(K)}\geq \frac{C''\lambda^{n_0}}{2\rho_+(K)}\geq\delta.
    	$$
    	
    	Since $x,\,y\in X$ are arbitrary, $f|_X$ is forward expansive.
    \end{proof}

\begin{remark}
	Note that $f|_X$ cannot be forward expansive if there exists a branch point in $X$.
\end{remark}



Next, we prove that a metric CXC system is $h$-expansive if and only if there is no branch point in the repellor.
Here is the main idea of the proof. Given a sufficiently large $m\in\N$, we construct recursively a sequence $\{A_n\}_{n\in\N}$ of subsets of $X$ such that each $A_n$ is contained in an element in $\bigvee_{i=0}^{n}f^{-i}(\cW_m)$. We prove by induction that for each sufficiently large $l\in\N$, the number of elements in $\bigvee_{j=0}^{n}f^{-j}(\cW_l)$ needed to cover $A_n$ is bounded from below, and such lower bound leads to the conclusion that $h(f|\cW_m)>0$.

We recall the following lemma about $h$-expansiveness; see \cite[Lemma~5.12]{Li15}.
\begin{lemma}[Li \cite{Li15}]\label{l:h-expansive can be iterated}
Let $g\:Z\rightarrow Z$ be a continuous map on a compact metric space $(Z,d)$. If $g$ is $h$-expansive then so is $g^n$ for each $n\in\N$.
\end{lemma}

\begin{theorem}\label{t:not h-expansive}
    Let $f\:(\fX_1,X)\rightarrow(\fX_0,X)$ and $\{\cU_n\}_{n\in\N_0}$ satisfy the Assumptions in Section~\ref{s:assumption}. 
    Suppose that $B_f\cap X\neq\emptyset$. Then $f|_X$ is not $h$-expansive.
\end{theorem}
\begin{proof}
	Consider $x_0\in B_f\cap X$.

    First, without loss of generality, we make some assumptions.
\begin{itemize}
	\smallskip
    \item[\textnormal{(i)}] By Lemma~\ref{l:h-expansive can be iterated}, it suffice to prove that $(f|_X)^n$ is not $h$-expansive for some $n\in\N$. By Lemma~\ref{l:iteration_satisfy_assumptions}, for each $n\in\N$, $f^n\:(\fX_n,X)\rightarrow(\fX_0,X)$ and $\{\cU_{nk}\}_{k\in\N_0}$ satisfy the Assumptions in Section~\ref{s:assumption}, and when $n$ is sufficiently large, $\cU_{n(k+1)}$ is a refinement of $\cU_{nk}$ for each $k\in\N_0$. Thus, we can assume that $\cU_{n+1}$ is a refinement of $\cU_n$ for each $n$.
    \smallskip
    \item[\textnormal{(ii)}] By [{\bf Deg}], for each $x\in X$, the forward orbit $\{f^n(x)\}_{n\in\N_0}$ contains at most $\lceil\log_2 p\rceil$ many branch points. Thus, there exists a branch point whose preimages under iterations of $f$ contain no branch point. We may assume that $\deg(f;x)=1$ for each $n\in\N$ and each $x\in f^{-n}(x_0)$. 
\end{itemize} 

    By Lemma~\ref{l:Wn_refining}, $\{\cW_n\}_{n\in\N_0}$ is a refining sequence of open covers of $X$. Thus, it follows from (\ref{e:H(A|B)}) and (\ref{e:top cond entropy}) in Definition~\ref{d:conditional_tail_entropy} that it suffices to prove that
    \begin{equation*}
        h(f_X|\cW_m)=
        \lim_{l\to +\infty} \lim_{n\to +\infty}
        \frac{1}{n}H
        \biggl(
        \bigvee_{i=0}^{n-1} f_X^{-i}(\cW_l)
        \,\bigg|\,
        \bigvee_{j=0}^{n-1} f_X^{-j}(\cW_m)
        \biggr)
        >0
    \end{equation*}
    for each $m\in\N$ sufficiently large.

    Let $k\=\deg(f;x_0)$. Then $k>1$. By {\bf [Expans]}, for each sufficiently large $m\in\N$, there exists $U_0^m\in\cU_m$ such that $x_0\in U_0^m$ and that $f|_{U_0^m}\:U_0^m\rightarrow f(U_0^{m})$ is an FBC map of degree $k$. 
    Fix such $m$ and $U_0^m$, and denote $U_0^{m-1}\=f(U_0^m)$.

    Fix $y_0\in U_0^{m-1}$ such that $y_0\in X\setminus \post(f)$. Such $y_0$ exists since $\post(f)$ is nowhere dense in $X$ (see Proposition~\ref{p:Pf nowhere dense}). Since 
    \begin{equation*}
    \sum_{y\in U_0^m,\,f(y)=y_0}\deg(f;y)=k,
    \end{equation*}
    there exists $k$ inverse images of $y_0$ under $f$ in $U_0^m$. We enumerate $f^{-1}(y_0)\cap U_0^m=\{z_1,\,z_2,\,\cdots,\,z_k\}$.
    By Proposition~\ref{p:roundness}~(i), we can choose a sufficiently large integer $m'\in\N$ and a level $m'$ good open set ${U_0^{m'}\in\cU_{m'}}$ such that the following properties hold:
    \begin{enumerate}
        \smallskip
        \item[\textnormal{(1)}] $y_0\in U_0^{m'}$ and $U_0^{m'}\subseteq U_0^{m-1}$.
        \smallskip
        \item[\textnormal{(2)}] For all $U,\,U'\in\cU_{m'+1}$, if $f(U)=f(U')=U_0^{m'}$ and $z_i\in U$, $z_j\in U'$ for some $i,\,j\in\{1,\,2,\,\cdots,\,k\}$ with $i\neq j$, then $d(U,U') \=\inf\{d(x,y):x\in U,\,y\in U'\}>0$.
    \end{enumerate}

    For each $i\in\{1,\,2,\,\cdots,\,k\}$, choose $\tU_i\in\cU_{m'+1}$ such that $f \bigl( \tU_i \bigr)=U_0^{m'}$ and that $z_i\in\tU_i$. Then $d \bigl( \tU_i, \tU_j \bigr)>0$ for all $1\leq i < j\leq k$. We denote
    \begin{equation}   \label{e:Pf_t_Intro_not h-expansive_delta}
    	\delta \=\min \bigl\{d \bigl(\tU_i,\tU_j \bigr) : 1\leq i < j\leq k \bigr\}>0.
    \end{equation}
    
    Fix an arbitrary natural number 
    \begin{equation*}
    	l>\log_\theta\frac{\delta}{C'd_0},
    \end{equation*}
    where $C'>0$, $d_0>0$, and $\theta\in(0,1)$ are the constants from Proposition~\ref{p:comparability}.
    
    By Proposition~\ref{p:repellors are fractal}, there exists $t\in\N$ and ${V^{m+t}\in\cU_{m+t}}$, such that $f^t(V^{m+t})=U_0^m$, and that $V^{m+t}\subseteq U_0^{m'}$. We denote $V^{m+r}\=f^{t-r}(V^{m+t})$ for each $r \in \{0,\, 1,\, \dots, \, t-1\}$. In particular, $V^m=U_0^m$.

    We will construct a sequence $\{A_n\}_{n\in\N_0}$ of subsets of $X$ such that the following conditions hold for each $n\in\N_0$:
    \begin{enumerate}
    	\smallskip
        \item[\textnormal{(1)}] $A_n\subseteq A$ for some $A\in\bigvee_{i=0}^{n}f^{-i}(\cU_m)$.
        \smallskip
        \item[\textnormal{(2)}] $f(A_n)=A_{n-1}$ if $n>0$.
        \smallskip
        \item[\textnormal{(3)}] $\card(A_n)=k^{\lfloor\frac{n}{t+1}\rfloor}$.
        \smallskip
        \item[\textnormal{(4)}] $A_n\subseteq V^{m+r}$ if $n=(t+1)s+r$, where $s\in\N_0$ and $r\in\{0,\,1,\,\dots,\,t\}$.
        \smallskip
        \item[\textnormal{(5)}] For $B\in\bigvee_{i=0}^n f^{-i}(\cU_l)$ and $x,\,y\in A_n$ with $x\neq y$, we have $\{x,\,y\}\nsubseteq B$.
    \end{enumerate}
    
    We now construct $A_n$ recursively.
    
    Let $A_0=\{x_0\}$. Clearly $A_0$ satisfies conditions~(1) through~(5).

    Make the induction hypothesis that $A_n$ is defined and satisfies conditions~(1) through~(5) for each $n\in\{0,\,1,\,\dots,\,(t+1)s+r\}$, where $s\in\N_0$ and $r\in\{0,\,1,\,\dots,\,t\}$. We construct $A_{(t+1)s+r+1}$ in the following cases depending on $r$, and show that $A_{(t+1)s+r+1}$ satisfies conditions~(1) through~(5) (with $n\=(t+1)s + r+1$).

    \textit{Case 1.} Assume $0\leq r<t$.

    By our induction hypothesis, $A_{(t+1)s+r}\in V^{m+r}$. For each $x\in A_{(t+1)s+r}$, choose one point $x'\in V^{m+r+1}$ such that $f(x')=x$. Define $A_{(t+1)s+r+1}$ to be the collection of all such $x'$ that correspond to $x\in A_{(t+1)s+r}$. 

    Clearly, conditions~(2) and~(4) are satisfied by definition. Now we verify conditions~(1), (3), and~(5).
    
    We first verify condition~(1). By our induction hypothesis, $A_{(t+1)s+r}\subseteq A$ for some $A\in\bigvee_{i=0}^{(t+1)s+r}f^{-i}(\cU_m)$. Suppose that $A=\bigcap_{i=0}^{(t+1)s+r}f^{-i}\bigl(\hU_i^m\bigr)$ for some $\hU_i^m\in\cU_m$ for each $i\in\{0,\,1,\,\dots,\,(t+1)s+r\}$. Since $A_{(t+1)s+r+1}\subseteq V^{m+r+1}$, and $\{\cU_n\}_{n\in\N_0}$ is a refining sequence, there exists $\hU^m\in\cU_m$ such that $A_{(t+1)s+r+1}\subseteq V^{m+r+1}\subseteq \hU^m$. Thus
    \begin{equation*}
        A_{(t+1)s+r+1}\subseteq\hU^m\cap\bigcap_{i=1}^{(t+1)s+r+1}f^{-i}\bigl(\hU_{i-1}^m\bigr)\in\bigvee_{i=0}^{(t+1)s+r+1}f^{-i}(\cU_m).
    \end{equation*}
    So, condition~(1) is verified.

    Next, we verify condition~(3). By definition, $\card\bigl(A_{(t+1)s+r+1}\bigr)=\card\bigl(A_{(t+1)s+r}\bigr)$. When $0\leq r<t$, we have $\bigl\lfloor\frac{(t+1)s+r+1}{t+1}\bigr\rfloor=s=\bigl\lfloor\frac{(t+1)s+r}{t+1}\bigr\rfloor$, and condition~(3) is verified.

    Then, we verify condition~(5). Suppose that there are $x,\,y\in A_{(t+1)s+r+1}$ with $x\neq y$, such that $\{x,\,y\}\subseteq B$ for some $B\in\bigvee_{i=0}^{(t+1)s+r+1}f^{-i}(\cU_l)$. Suppose $B=\bigcap_{i=0}^{(t+1)s+r+1}f^{-i}\bigl(\hU_i^l\bigr)$ for some $\hU_i^l\in\cU_l$ for each $i\in\{0,\,1,\,\cdots,\,(t+1)s+r+1\}$.
    By definition $f(x),\,f(y)\in A_{(t+1)s+r}$ and $f(x)\neq f(y)$. However, 
    \begin{equation*}
    	\{f(x),\,f(y)\}\subseteq\bigcap_{i=0}^{(t+1)s+r}f^{-i}\bigl(\hU_{i+1}^l\bigr)\in\bigvee_{i=0}^{(t+1)s+r}f^{-i}(\cU_l),
    \end{equation*}
    which contradicts the induction hypothesis. So, condition~(5) is verified.
    
    \smallskip
    \textit{Case 2.} Assume $r=t$, then $(t+1)s+r+1=(t+1)(s+1)$.
    
    By our induction hypothesis $A_{(t+1)s+t}\subseteq V^{m+t}\subseteq U_0^{m'}$. Fix an arbitrary number $i\in\{1,\,2,\,\dots,\,k\}$. For each $x\in A_{(t+1)s+t}$, choose one point $x'\in \tU_i$ such that $f(x')=x$. Define $A^i_{(t+1)(s+1)}$ to be the collection of all such $x'$ corresponding to some $x\in A_{(t+1)s+t}$. 
    
    Then we define $A_{(t+1)(s+1)}\=\bigcup_{i=1}^k A^i_{(t+1)(s+1)}$.

    Clearly, conditions~(2) and~(4) are satisfied by definition, and condition~(1) can be verified by the same process as in \textit{Case 1.}  It suffice to verify conditions~(3) and~(5).

    First, we verify condition~(3). By construction, $\card\bigl(A^i_{(t+1)(s+1)}\bigr)=\card\bigl(A_{(t+1)s+t}\bigr)$. Since $A_{(t+1)(s+1)}^i\subseteq \tU_i$ for each $i\in\{1,\,2,\,\dots,\,k\}$, and $\tU_i\cap\tU_j=\emptyset$ for each pair $i,\,j\in\{1,\,2,\,\dots,\,k\}$ with $i\neq j$, we have
    \begin{equation*}
    	\card \bigl( A_{(t+1)(s+1)} \bigr)=k\cdot\card \bigl( A_{(t+1)s+t)} \bigr) = k^{\lfloor\frac{(t+1)s+t}{t+1}\rfloor+1}=k^{s+1},
    \end{equation*}
    and condition~(3) is verified.

    Next, we verify condition~(5). Suppose that there are $x,\,y\in A_{(t+1)(s+1)}$ with $x\neq y$ such that $\{x,\,y\}\subseteq B$ for some $B\in\bigvee_{a=0}^{(t+1)(s+1)}f^{-a}(\cU_l)$. Suppose $B\subseteq U^l$ for some $U^l\in\cU_l$. By Proposition~\ref{p:comparability}~(i), $\diam(U^l)\leq C'd_0\theta^l$, so $d(x,\,y)<\delta$ since $l>\log_\theta\frac{\delta}{C'd_0}$. Then by (\ref{e:Pf_t_Intro_not h-expansive_delta}) there is some $i\in\{1,\,2,\,\dots,\,k\}$ such that $\{x,\,y\}\subseteq \tU_i$, and by construction $f(x)\neq f(y)$. However, since $\{x,\,y\}\subseteq B$ and $B\in\bigvee_{a=0}^{(t+1)(s+1)}f^{-a}(\cU_l)$, we have
    $\{f(x),\,f(y)\}\subseteq S$ for some $S\in\bigvee_{a=0}^{(t+1)s+t}f^{-a}(\cU_l)$, which contradicts condition~(5) for $A_{(t+1)s+t}$ in the induction hypothesis.

    The recursive construction and the inductive proof of the conditions are now complete. We note that $A_n$ is a subset of $X$ for each $n\in\N_0$.

    \smallskip
     The following properties straightforwardly follow from conditions~(1), (3), and~(5).
    \begin{enumerate}
    	\smallskip
        \item[\textnormal{(i)}] $A_n\subseteq A$ for some $A\in\bigvee_{i=0}^{n}f_X^{-i}(\cW_m)$.
        \smallskip
        \item[\textnormal{(ii)}] $\card(A_n)=k^{\lfloor\frac{n}{t+1}\rfloor}$.
        \smallskip
        \item[\textnormal{(iii)}] For all $B\in\bigvee_{i=0}^{n}f_X^{-i}(\cW_l)$ and $x,\,y\in A_n$ with $x\neq y$, we have $\{x,\,y\}\nsubseteq B$.
    \end{enumerate}

    So for such $A\in\bigvee_{i=0}^{n}f_X^{-i}(\cW_m)$ that contains $A_n$, if $\cB\subseteq\bigvee_{i=0}^{n}f_X^{-i}(\cW_l)$ satisfies $A\subseteq\bigcup\cB$, then $\card(\cB)\geq\card(A_n)=k^{\lfloor\frac{n}{t+1}\rfloor}$.

    As a consequence, for each sufficiently large $m\in\N$, there exists $t_m\in\N$, such that for each sufficiently large $l\in\N$ and each $n\in\N$, by (\ref{e:H(A|B)}) in Definition~\ref{d:conditional_tail_entropy}, we have 
    \begin{equation}   \label{e:Pf_t_Intro_not h-expansive_H_lower_bound}
        H\biggl(
        \bigvee_{i=0}^{n-1} f_X^{-i}(\cW_l)
        \,\bigg|\,
        \bigvee_{j=0}^{n-1} f_X^{-j}(\cW_m)
        \biggr)
        \geq
        \log k^{\lfloor\frac{n-1}{t_m+1}\rfloor}.
    \end{equation}
    Thus for each sufficiently large $m\in\N$, by (\ref{e:top cond entropy}) in Definition~\ref{d:conditional_tail_entropy} and (\ref{e:Pf_t_Intro_not h-expansive_H_lower_bound}), we have
    \begin{equation*}
    \begin{aligned}
        h(f_X|\cW_m)
        =&
        \lim_{l\to +\infty} \lim_{n\to +\infty}
        \frac{1}{n}H\biggl(
        \bigvee_{i=0}^{n-1} f_X^{-i}(\cW_l)
        \,\bigg|\,
        \bigvee_{j=0}^{n-1} f_X^{-j}(\cW_m)
        \biggr)\\
        \geq&
        \liminf_{l\to +\infty} \liminf_{n\to +\infty}
        \frac{1}{n}\log k^{\lfloor\frac{n-1}{t_m+1}\rfloor}
        =\frac{\log k}{t_m+1}>0.
    \end{aligned}
    \end{equation*}
    Therefore, the map $f|_X$ is not $h$-expansive, and the proof is complete.
\end{proof}

	Here we give a quick proof of Theorem~\ref{t:Intro_expansive and h-expansive}.
\begin{proof}[Proof of Theorem~\ref{t:Intro_expansive and h-expansive}]
	The proof follows from Remark~\ref{r: expans=>h-expans=>asym h-expans}, Theorems~\ref{t:expansive} and \ref{t:not h-expansive}.
	
	\smallskip
	
	(i)$\Rightarrow$(ii): When there is no branch point in $X$, by Theorem~\ref{t:expansive}, $f|_X$ is forward expansive.
	
	\smallskip
	
	(ii)$\Rightarrow$(iii): By Remark~\ref{r: expans=>h-expans=>asym h-expans}, $f|_X$ is $h$-expansive whenever $f|_X$ is forward expansive.
	
	
	\smallskip
	
	(iii)$\Rightarrow$(i): Assume on the contrary that $X\cap B_f\neq\emptyset$. Then, by Theorem~\ref{t:not h-expansive}, $f|_X$ is not $h$-expansive. Therefore, $X\cap B_f=\emptyset$ if $f|_X$ is $h$-expansive.
\end{proof}

 
\section{Equilibrium states for metric CXC systems}\label{s:equilibrium states}
In this section, we will establish the existence of an equilibrium state for a metric CXC system and a continuous potential.

\subsection{Basic concepts}
We start with a brief review of some concepts from dynamics. We refer the reader to \cite{PU10} for a more detailed study of these concepts.

Let $(Z,d)$ be a compact metric space and $g\:Z\rightarrow Z$ a continuous map. For each $n\in\N$ the \defn{dynamical distance} is defined as 
\begin{equation*}
    d_g^n(x,y)\=\max\bigl\{d \bigl( g^k(x),g^k(y) \bigr) : k\in\{0,\,1,\,\dots,\,n-1\}\bigr\},
\end{equation*}
for $x,y\in Z$. We can check that $d_g^n$ is a metric on $Z$. A set $F\subseteq Z$ is \defn{$(n,\epsilon)$-separated} for $n\in\N$ and $\epsilon>0$ if for each pair of distinct points $x,y\in F$, we have $d_g^n(x,y)>\epsilon$. For $n\in\N$ and $\epsilon>0$, we use $F_n(\epsilon)$ to denote a maximal (in the sense of inclusion) $(n,\epsilon)$-separated set.

For each real-valued continuous function $\psi\in C(Z)$, the following limits exist
\begin{equation*}
        P(g,\psi)
        \=\lim_{\epsilon\to 0}\limsup_{n\to +\infty}\frac{1}{n}\log\sum_{x\in F_n(\epsilon)}\exp{(S_n\psi(x))},
\end{equation*}
where $S_n\psi(x) =\sum_{j=0}^{n-1}\psi(g^j(x))$. Note that $P(g,\psi)$ above is independent of the choices of $F_n(\epsilon)$ and is independent of $d$ as long as the topology defined by $d$ remains the same. We call $P(g,\psi)$ the \defn{topological pressure} of $g$ with respect to potential $\psi$, and the quantity $h_{\top}(g) \=P(g,0)$ is called the \defn{topological entropy} of $g$. 

A \defn{measurable partition} $\cA$ of $Z$ is a finite or countably infinite collection $\cA=\{A_j:j\in J\}$ of mutually disjoint Borel sets with $\bigcup\cA=Z$, where $J$ is a finite or countably infinite index set. 
For $x\in Z$ we denote by $\cA(x)$ the unique element of $\cA$ such that $x\in\cA(x)$. 
Let $\cA=\{A_j:j\in J\}$ and $\cB=\{B_k:k\in K\}$ be measurable partitions of $Z$, we say $\cA$ is a \defn{refinement} of $\cB$ if for each $A_j\in\cA$ there exist $B_k\in\cB$ with $A_j\subseteq B_k$. 
The \defn{common refinement} of $\cA$ and $\cB$ is defined as $\cA\vee\cB \=\{A_j\cap B_k:A_j\in\cA,\ B_k\in\cB\}$, which is also a measurable partition. 
For a continuous map $g\:Z\rightarrow Z$, we set $g^{-1}(\cA) \=\{g^{-1}(A_j):A_j\in\cA\}$, and denote for $n\in\N$,
\begin{equation*}
    \cA_g^n \=\bigvee_{j=0}^{n-1}g^{-j}(\cA)=\cA\vee g^{-1}(\cA)\vee\dots\vee g^{-(n-1)}(\cA).
\end{equation*}

Let $\cP(Z)$ denote the set of Borel probability measures and $\cM(Z,g)$ denote the set of $g$-invariant Borel probability measures on $Z$, by Krylov--Bogolyubov theorem $\cM(Z,g)\neq\emptyset$. 
The entropy of a measurable partition $\cA$ is
\begin{equation*}
     H_\mu(\cA)
     \=\sum_{j\in J} -\mu(A_j)\log(\mu(A_j))
\end{equation*}
where $0\log 0$ is defined to be $0$. If $H_\mu(\cA)<+\infty$, then the limit $h_\mu(g,\cA)\=\lim\limits_{n\to +\infty}\frac{1}{n}H_\mu\bigl(\cA_g^n\bigr)$ exists,
and the \defn{measure-theoretic entropy} of $g$ for $\mu$ is given by
\begin{equation*}
    h_\mu(g)\=\sup\{h_\mu(g,\cA)\},
\end{equation*}
where $\cA$ ranges over all measurable partitions with $h_\mu(g,\cA)<+\infty$.
For each $\psi\in C(Z)$, the \defn{measure-theoretic pressure} $P_\mu(g,\psi)$ of $g$ for the measure $\mu$ and potential $\psi$ is
\begin{equation*}
    P_\mu(g,\psi) \=h_\mu(g)+\int\!\psi\diff \mu.
\end{equation*}

By the Variational Principle, we have for each $\psi\in C(Z)$
\begin{equation}\label{e:VP1}
    P(g,\psi)=\sup\{P_\mu(g,\psi):\mu\in\cM(Z,g)\}.
\end{equation}
In particular, when $\psi$ is the constant function $0$, we have
\begin{equation}\label{e:VP2}
    h_{\top}(g)=\sup\{h_\mu(g):\mu\in\cM(Z,g)\}.
\end{equation}
A measure $\mu\in\cM(Z,g)$ that attains the supremum in (\ref{e:VP1}) is called an \defn{equilibrium state} for the transformation $g$ and potential $\psi$. A measure $\mu\in\cM(Z,g)$ that attains the supremum in (\ref{e:VP2}) is called a \defn{measure of maximal entropy} for $g$.

\subsection{Existence of equilibrium states} 
In this subsection, we show that equilibrium states exist for metric CXC systems and continuous potentials.

Let $Z$ be a compact metric space and $g\:Z\rightarrow Z$ a continuous map. First, we recall that $\cM(Z,g)$ is weak$^*$ compact. We refer the reader to \cite{PU10} for more details.


The following proposition gives a consequence of the asymptotic $h$-expansiveness. We refer the reader to \cite{Mi76} for more details.
\begin{prop}\label{p:upper semi continuous}
    Let $g\:Z\rightarrow Z$ be a continuous map on a compact metric space $Z$. If $g$ is asymptotically $h$-expansive, then the map
    \begin{equation*}
        h_\centerdot (g)\:\cM(Z,g)\rightarrow\R,\ \mu\mapsto h_\mu(g)
    \end{equation*}
    is upper semi-continuous with respect to the weak$^*$ topology.
\end{prop}

Finally, we establish the existence of equilibrium states for metric CXC systems.

\begin{proof}[Proof of Theorem~\ref{t:Intro existence equilibrium}] 
    Let $f\:(\fX_1,X)\rightarrow(\fX_0,X)$ and $\{\cU_0\}_{n\in\N_0}$ satisfy the Assumptions in Section~\ref{s:assumption}, and $\phi\in C(X)$. Note that $X$ is compact.

    By Proposition~\ref{p:upper semi continuous} and Theorem~\ref{t:Intro asym h-expans}, the map
    $h_\centerdot (f_X)\:\cM(X,f_X)\rightarrow\R,\ \mu\mapsto h_\mu(f_X)$
    is upper semi-continuous. Then since $\phi\in C(X)$, the map
     \begin{equation*}
        P_\centerdot (f_X,\phi)\:\cM(X,f_X)\rightarrow\R,\ \mu\mapsto P_\mu(X,f_X)=h_\mu(f_X)+\int\! \phi\diff \mu
    \end{equation*}
    is upper semi-continuous.    
    Since $\cM(X,f_X)$ is weak$^*$ compact, there exists $\nu\in\cM(X,f_X)$ that attains the supremum in
        $\sup\{P_\mu(f_X,\phi):\mu\in\cM(X,f_X)\}$.        
    By the Variational Principle, such $\nu\in\cM(X,f_X)$ is an equilibrium state.
\end{proof}
In particular, let $\phi$ be the constant function $0$, then we obtain the existence of measures of maximal entropy.
\begin{cor}
   Let $f\:(\fX_1, X)\rightarrow(\fX_0, X)$ and $\{\cU_n\}_{n\in\N_0}$ satisfy the Assumptions in Section~\ref{s:assumption}, then there exists at least one measure of maximal entropy for the map $f|_X$.
\end{cor}

\section{Level-$2$ large deviation principle and equidistribution}\label{s:LD}
In this section, we show a consequence of our previous results, namely, a level-$2$ large deviation principle for iterated preimages under additional assumptions, from which we obtain an equidistribution result.
\subsection{Level-$2$ large deviation principle}\label{ss:level-2 LD}
Let $Z$ be a compact metric space, and $\cP(Z)$ be the set of Borel probability measures on $Z$ equipped with the weak$^*$ topology. Note this topology is metrizable (see for example, \cite[Theorem~5.1]{Co85}). Let $I\:\cP(Z)\rightarrow [0,+\infty]$ be a lower semi-continuous function.

A sequence $\{\Omega_n\}_{n\in\N}$ of Borel probability measures on $\cP(Z)$ is said to satisfy a \defn{level-$2$ large deviation principle with rate function $I$} if for each closed subset $\fF$ of $\cP(Z)$ and each open subset $\fG$ of $\cP(Z)$ we have
\begin{equation*}
\begin{aligned}
    \limsup_{n\to+\infty}\frac{1}{n}\log (\Omega_n(\fF))&\leq-\inf\{I(x)\:x\in\fF\},\\
    \liminf_{n\to+\infty}\frac{1}{n}\log (\Omega_n(\fG))&\geq-\inf\{I(x)\:x\in\fG\}.
\end{aligned}
\end{equation*}

We state a theorem due to Y.~Kifer \cite[Theorem~4.3]{Ki90}, reformulated by H.~Comman and J.~Rivera-Letelier \cite[Theorem~C]{CRL11}. 

\begin{theorem}[Kifer \cite{Ki90}, Comman  \& Rivera-Letelier \cite{CRL11}]
\label{t:LD, Kifer and Juan}
    Let $Z$ be a compact metrizable space, and $g\:Z\rightarrow Z$ be a continuous map. Fix $\phi\in C(Z)$, and let $H$ be a dense vector subspace of $C(Z)$ with respect to the uniform norm. Let $I^\phi\:\cP(Z)\rightarrow [0,+\infty]$ be the function defined by
    \begin{equation*}
    I^\phi(\mu) \=
    \begin{cases}
        P(g,\phi)-\int\!\phi\diff \mu-h_\mu(g),&\mbox{ if }\mu\in\cM(Z,g);\\
        +\infty, &\mbox{ if } \mu\in\cP(Z)\setminus \cM(Z,g).
    \end{cases}
    \end{equation*}

    We assume the following conditions are satisfied:
    \begin{enumerate}
    	\smallskip
        \item[\textnormal{(i)}] The function $h_\centerdot(g)\:\cM(Z,g)\rightarrow\R$, $\mu\mapsto h_\mu(g)$, is finite and upper semi-con\-tin\-u\-ous.
        \smallskip
        \item[\textnormal{(ii)}] For each $\psi\in H$, there exists a unique equilibrium state for $g$ and potential $\phi+\psi$.
    \end{enumerate}
    Let $\{\Omega_n\}_{n\in\N}$ be a sequence of Borel probability measures on $\cP(Z)$ with the property that for each $\psi\in H$, 
    \begin{equation*}
        \lim_{n\to+\infty}\frac{1}{n}\log\int_{\cP(Z)}\!\exp{\left(n\int\!\psi \diff\mu\right)} \diff\Omega_n(\mu)=P(g,\phi+\psi)-P(g,\phi).
    \end{equation*}
    Then $\{\Omega_n\}_{n\in\N}$ satisfies a level-$2$ large deviation principle with rate function $I^\phi$, and it converges in the weak$^*$ topology to the Dirac measure supported on the unique equilibrium state for $g$ and potential $\phi$. Furthermore, for each convex
    open subset $\fG$ of $\cP(Z)$ containing some invariant measure, we have
    \begin{equation*}
        \lim_{n\to+\infty}\frac{1}{n}\log(\Omega_n(\fG))
        =\lim_{n\to+\infty}\frac{1}{n}\log(\Omega_n(\closure\fG))
        =-\inf_\fG I^\phi=-\inf_{\closure\fG} I^\phi.
    \end{equation*}
\end{theorem}

\subsection{The Additional Assumptions}\label{ss:LD addassumptions}
We state below the assumptions under which we will establish a level-$2$ large deviation principle.  We will repeatedly refer to such assumptions later.

Before we state these assumptions, we recall the notion of \defn{strongly path-connectedness}. Here we adopt \cite[Definition~2.3]{DPTUZ21}.
\begin{definition}[Strongly path-connected]
    A topological space $\fX$ is \defn{strongly path-connected} if, for any finite or countably infinite subset $S$ of $\fX$, the space $\fX\setminus S$ is path-connected.
\end{definition}
\begin{addassumptions}
\quad
\begin{enumerate}
	\smallskip
    \item[\textnormal{(1)}] $\fX_0, \fX_1$ are strongly path-connected.
    \smallskip
    \item[\textnormal{(2)}] the branch set $B_f$ is finite. 
\end{enumerate}
\end{addassumptions}
 Note that these additional assumptions were adopted in \cite{DPTUZ21}.

\subsection{Proof of Theorem~\ref{t:Intro_LD}}\label{ss:pf of LD}
The goal of this subsection is to establish Theorem~\ref{t:Intro_LD}. We first recall the notion of geometric coding tree from \cite{DPTUZ21} and then use it to establish a characterization of the topological pressure. With such a characterization, we can apply Theorem~\ref{t:LD, Kifer and Juan} to prove Theorem~\ref{t:Intro_LD}.

In the following context, we assume that $f\:(\fX_1,X)\rightarrow(\fX_0,X)$ and $\{\cU_n\}_{n\in\N_0}$ satisfy the Assumptions in Section~\ref{s:assumption} and the Additional Assumptions in Subsection~\ref{ss:LD addassumptions}, and let $\phi$ be a real-valued H\"older continuous function on $X$.

For each $x\in X$ and each $n\in\N$, put 
\begin{equation*}
	W_n(x) \=\frac{1}{n}\sum\limits_{i=0}^{n-1}\delta_{f^i(x)}.
\end{equation*}
For a given sequence $\{x_n\}_{n\in\N}$ of points in $ X$, consider the following sequence $\{\Omega_n\}_{n\in\N}$ of Borel probability measures on $\cP(X)$.

{\bf Iterated Preimages:}
\begin{equation*}
    \Omega_n\=
    \sum\limits_{y\in f^{-n}(x_n)}
    \frac{\deg(f^n;y)\cdot\exp{(S_n\phi(y))}}
    {\sum_{z\in f^{-n}(x_n)}\deg(f^n;z)\cdot\exp{(S_n\phi(z))}}\delta_{W_n(y)},
\end{equation*}
where $S_n\phi(x) =\sum_{i=0}^{n-1}\phi(f^i(x))$.
\begin{rem}
    It is clear that the quantities $\sum_{i=0}^{n-1}\phi(f^i(x))$ and $\sum_{i=0}^{n-1}\phi(f_X^i(x))$ are identical. By Proposition~\ref{p:same loc deg}, the quantities $\deg(f^n;x)$ and $\deg(f_X^n;x)$ are identical. Thus, we do not distinguish these quantities in the following context.
\end{rem}

Our goal in this section is to establish a level-$2$ large deviation principle for $\{\Omega_n\}_{n\in\N}$ with rate function $I^\phi$ given in Theorem~\ref{t:LD, Kifer and Juan}. We do this by applying Theorem~\ref{t:LD, Kifer and Juan}. It suffices to verify that the conditions in Theorem~\ref{t:LD, Kifer and Juan} are satisfied.

We first verify that for each $\psi\in C^{0,\alpha}(X)$, 
\begin{equation}\label{e:to verify LD}
    \lim_{n\to+\infty}\frac{1}{n}\log\int_{\cP(X)} \! \exp{\left(n\int\!\psi \diff\mu\right)} \diff\Omega_n(\mu)=P(f_X,\phi+\psi)-P(f_X,\phi).
\end{equation}

By results in \cite{DPTUZ21}, under the Additional Assumption in Subsection~\ref{ss:LD addassumptions}, the topological pressure of a metric CXC system and the shift map on a space of one-sided sequences can be related through a semiconjugacy. We will use this idea to verify (\ref{e:to verify LD}).

\subsubsection{Geometric coding tree and semiconjugacy}
We construct a semiconjugacy from a shift map to $f$ as in \cite{DPTUZ21}. The construction is based on the idea of ``geometric coding tree." We refer the reader to \cite{DPTUZ21} for more details.

Let $\Sigma\=\{1,\,\dots,\,d\}^\N$ be the space of infinite one-sided sequences on $d$ symbols, and $\sigma\:\Sigma\rightarrow\Sigma$ the shift map. We equip $\Sigma$ with the standard metric $\rho(\xi,\eta)\=2^{-\min\{k\geq 0\,:\,\xi_k\neq\eta_k\}}$ for distinct $\xi = (\xi_0,\xi_1,\dots,)$ and $\eta = ( \eta_0, \eta_1, \dots)$ in $\Sigma$.

    Suppose that $f\:(\fX_1,X)\rightarrow(\fX_0,X)$ and $\{\cU_n\}_{n\in\N_0}$ satisfy the Assumptions in Section~\ref{s:assumption} and the Additional Assumptions in Subsection~\ref{ss:LD addassumptions}. Pick $w\in X\setminus \post(f)$, let $w_1,w_2,\dots,w_d$ be all of its preimages under $f$. For each $w_i$, choose a path $\gamma_i$ in $\fX_0$ connecting $w$ to $w_i$ and avoiding $\post(f)$. Such paths exist since $\fX_0$ is strongly path-connected. Here, by a \emph{path} $\gamma$ in $\fX_0$ we mean a continuous map $\gamma \: [0,1]  \rightarrow \fX_0$.

We recall the path-lifting property (see for example, \cite[Lemma~2.15]{DPTUZ21}). Such path-lifting property will be used in the construction of the geometric coding tree.
\begin{prop}\label{p:path lift in DPTUZ21}
    Let $h\:Y\rightarrow Z$ be a finite branched covering map, and $\gamma$ be a continuous path in $Z$ that is disjoint from the set of branch values $V_h$. Let $x=\gamma(0)$, and $\tx\in h^{-1}(x)$. Then there exists a continuous path $\widetilde{\gamma}$ in $Y$ such that $h(\widetilde{\gamma})=\gamma$ and $\widetilde\gamma(0)=\tx$.
\end{prop}

Actually, such a lift is unique. We skip the proof since it is standard.
\begin{lemma}\label{l:uniqueness of path lift}
    Let $h\:Y\rightarrow Z$ be a finite branched covering map, and $\gamma$ be a continuous path in $Z$ that is disjoint from the set of branch values $V_h$. Let $x=\gamma(0)$, and $\tx\in h^{-1}(x)$. Let $\tgamma_1$ and $\tgamma_2$ be two paths in $Y$ such that $h(\tgamma_i)=\gamma$ and $\tgamma_i(0)=\tx$, $i\in\{1,\,2\}$, then $\tgamma_1=\tgamma_2$.
\end{lemma}
    

     

For each sequence $\xi=(\xi_0,\,\xi_1,\,\dots)\in\Sigma$, we define a sequence $\{z_n(\xi)\}_{n\in\N_0}$ of points in $X$ inductively as follows. Put $z_0(\xi)\=w_{\xi_0}$ and $\gamma_0(\xi)\=\gamma_{\xi_0}$. For each $n\in\N$, suppose that $z_{n-1}(\xi)$ is defined, let $\gamma_n(\xi)$ be a curve which is the branch of $f^{-n}(\gamma_{\xi_n})$ such that one of its ends is $z_{n-1}(\xi)$. Such lifts exist by Proposition~\ref{p:path lift in DPTUZ21}, and are well-defined by Lemma~\ref{l:uniqueness of path lift}. Define $z_n(\xi)$ as the other end of $\gamma_n(\xi)$.

It can be checked that when $f\:(\fX_1,X)\rightarrow(\fX_0,X)$ and $\{\cU_n\}_{n\in\N_0}$ satisfy the Assumptions in Section~\ref{s:assumption} and the Additional Assumptions in Subsection~\ref{ss:LD addassumptions}, the conditions in \cite[Theorem~1.1, Proposition~2.16, and Lemma~4.2]{DPTUZ21} are satisfied. Thus, we obtain what follows.

By the proof of \cite[Proposition~2.16]{DPTUZ21}, we have the following proposition, which gives a semiconjucacy from the shift map to $f$.

\begin{prop}\label{p:semi conjugacy}
Suppose that $f\:(\fX_1,X)\rightarrow(\fX_0,X)$ and $\{\cU_n\}_{n\in\N_0}$ satisfy the Assumptions in Section~\ref{s:assumption} and the Additional Assumptions in Subsection~\ref{ss:LD addassumptions}.
For each $\xi\in\Sigma$, let $z_n(\xi)$ be the sequence in $X$ defined as above. Then the following statements hold:
\begin{enumerate}
    \smallskip
    \item[\textnormal{(i)}] The limit $\lim_{n\to+\infty} z_n(\xi)$ exists in $X$. Define $\pi(\xi')\=\lim_{n\to+\infty} z_n(\xi')$ for each $\xi' \in \Sigma$.
    \smallskip
    \item[\textnormal{(ii)}] $f\circ\pi=\pi\circ\sigma$.
    \smallskip
    \item[\textnormal{(iii)}] The map $\pi\:\Sigma\rightarrow X$ is surjective and H\"older continuous.
\end{enumerate}
\end{prop}

By \cite[Lemma~4.2]{DPTUZ21}, we have the following proposition.

\begin{prop}[Das, Przytycki, Tiozzo, Urba\'nski, and Zdunik \cite{DPTUZ21}]  \label{p:top pressure}
    Suppose that $f\:(\fX_1,X)\rightarrow(\fX_0,X)$ and $\{\cU_n\}_{n\in\N_0}$ satisfy the Assumptions in Section~\ref{s:assumption} and the Additional Assumptions in Subsection~\ref{ss:LD addassumptions}. Denote $f_X=f|_X$.
    Let $\phi\: X \rightarrow \R$ be a H\"older continuous function, and $\pi$ be the semiconjugacy given in Proposition~\ref{p:semi conjugacy}. Then 
    \begin{equation*}
    	P(f_X,\phi)=P(\sigma,\phi\circ\pi).
    \end{equation*}
\end{prop}

The following proposition follows from \cite[Theorem~1.1]{DPTUZ21}.

\begin{prop}[Das, Przytycki, Tiozzo, Urba\'nski, and Zdunik \cite{DPTUZ21}] \label{p:Uniqueness equilibrium state}
    Suppose that $f\:(\fX_1,X)\rightarrow(\fX_0,X)$ and $\{\cU_n\}_{n\in\N_0}$ satisfy the Assumptions in Section~\ref{s:assumption} and the Additional Assumptions in Subsection~\ref{ss:LD addassumptions}.
    Let $\phi$ be a real-valued H\"older continuous function on $X$. Then there exists a unique equilibrium state $\mu_\phi$ for the map $f|_X$ and potential $\phi$.
\end{prop}

Then we establish some technical lemmas.

\begin{lemma}\label{l: z_m iterated}
    Suppose that $f\:(\fX_1,X)\rightarrow(\fX_0,X)$ and $\{\cU_n\}_{n\in\N_0}$ satisfy the Assumptions in Section~\ref{s:assumption} and the Additional Assumptions in Subsection~\ref{ss:LD addassumptions}. 
    Let $\xi=(\xi_0,\,\xi_1,\,\cdots)$ be in $\Sigma$. Then for each pair of $m,\,k\in\N$ with $k\leq m$, we have $f^k(z_m(\xi))=z_{m-k} \bigl( \sigma^k(\xi) \bigr)$. 
\end{lemma}
\begin{proof}
    We prove $f(z_m(\xi))=z_{m-1}(\sigma(\xi))$ for all $m\in\N$ by induction.
    
    First, consider the case where $m=1$. Then $\gamma_1(\xi)$ is the branch of $f^{-1}(\gamma_{\xi_1})$ such that one of its ends is $z_0(\xi)=w_{\xi_0}$. Then 
    \begin{equation*}
        f(z_1(\xi))=f(\gamma_1(\xi)(1))=\gamma_{\xi_1}(1)=w_{\xi_1}= z_0(\sigma(\xi)).
    \end{equation*}

    For an integer $m>2$, we make the induction hypothesis that $f(z_{m-1}(\xi))=z_{m-2}(\sigma(\xi))$. Since $\gamma_m(\xi)$ is the branch of $f^{-m}(\gamma_{\xi_m})$ such that one of its end is $z_{m-1}(\xi)$, we know that $f(\gamma_m(\xi))$ is the branch of $f^{-(m-1)}(\gamma_{\xi_m})$ such that one of its ends is $f(z_{m-1}(\xi))=z_{m-2}(\sigma(\xi))$. Thus $f(\gamma_m(\xi))$ is identical to $\gamma_{m-1}(\sigma(\xi))$, and $f(z_m(\xi))$ is identical to $z_{m-1}(\sigma(\xi))$.

    By induction, $f(z_m(\xi))=z_{m-1}(\sigma(\xi))$ holds for all $m\in\N$. Then it follows immediately that $f^k(z_m(\xi))=z_{m-k} \bigl( \sigma^k(\xi) \bigr)$ holds for each pair of $m,\,k\in\N$ with $k\leq m$, and the proof is complete.
\end{proof}

\begin{lemma}\label{l: Zm(a) neq Zm(b)}
    Suppose that $f\:(\fX_1,X)\rightarrow(\fX_0,X)$ and $\{\cU_n\}_{n\in\N_0}$ satisfy the Assumptions in Section~\ref{s:assumption} and the Additional Assumptions in Subsection~\ref{ss:LD addassumptions}.
    Let $\xi=(\xi_0,\,\xi_1,\,\cdots)$ and $\eta=(\eta_0,\,\eta_1,\,\cdots)$ be in $\Sigma$. If $\xi_n\neq\eta_n$ for some $n\in\N_0$, then $z_m(\xi)\neq z_m(\eta)$ for each $m\in\N$ with $m\geq n$.
\end{lemma}
\begin{proof}
    Assume that $\xi_n\neq\eta_n$ for some $n\in\N_0$. We prove that $z_m(\xi)\neq z_m(\eta)$ for all $m\geq n$ by induction.
    
    First, consider the case where $m=n$. Since $f^n(z_n(\xi))=w_{\xi_n}$, $f^n(z_n(\eta))=w_{\eta_n}$ and $\xi_n\neq\eta_n$, it is clear that $z_n(\xi)\neq z_n(\eta)$.

    For an integer $m>n$, we make the induction hypothesis that $z_{m-1}(\xi)\neq z_{m-1}(\eta)$. Suppose $z_m(\xi)=z_m(\eta)$. Then $w_{i_m}=f^m(z_m(\xi))=f^m(z_m(\eta))$ for some $i_m \in\{1,\,2,\,\dots,\,d\}$. Then $\gamma_m(\xi)$ and $\gamma_m(\eta)$ are two paths such that $f^m(\gamma_m(\xi))=f^m(\gamma_m(\eta))=\gamma_{i_m}$ and $\gamma_m(\xi)(1)=z_m(\xi)=z_m(\eta)=\gamma_m(\eta)(1)$. By Lemma~\ref{l:uniqueness of path lift}, $\gamma_m(\xi)=\gamma_m(\eta)$. Thus $z_{m-1}(\xi)=z_{m-1}(\eta)$, which contradicts to our induction hypothesis. So $z_m(\xi)\neq z_m(\eta)$.

    By induction, $z_m(\xi)\neq z_m(\eta)$ holds for all $m\geq n$.
\end{proof}

\begin{lemma}\label{l:a bijection between sequence and preimages}
    Suppose that $f\:(\fX_1,X)\rightarrow(\fX_0,X)$ and $\{\cU_n\}_{n\in\N_0}$ satisfy the Assumptions in Section~\ref{s:assumption} and the Additional Assumptions in Subsection~\ref{ss:LD addassumptions}. Let $\pi$ be the semiconjugacy given in Proposition~\ref{p:semi conjugacy}.  
    Then for all $x_0\in X\setminus \post(f)$, $\omega_0\in\Sigma$, and $n\in\N$ with $x_0=\pi(\omega_0)$,  the map $\pi|_{\sigma^{-n}(\omega_0)}\:\sigma^{-n}(\omega_0)\rightarrow f^{-n}(x_0)$ is bijective.
\end{lemma}
\begin{proof}
Fix arbitrary $x_0\in X\setminus \post(f)$, $\omega_0\in\Sigma$, $n\in\N$, and $\xi,\eta\in \sigma^{-n}(\omega_0)$ with $\xi\neq\eta$ and $x_0=\pi(\omega_0)$ (see Proposition~\ref{p:semi conjugacy}). By Lemma~\ref{l: z_m iterated}, for each integer $m\geq n$, it holds that $f^n(z_m(\xi))=z_{m-n}(\omega_0)=f^n(z_m(\eta))$. On the other hand, since $\xi\neq\eta$, there exists $k\in\{0,1,\cdots,n-1\}$ such that $\xi_k\neq\eta_k$. Then by Lemma~\ref{l: Zm(a) neq Zm(b)}, we know that $z_m(\xi)\neq z_m(\eta)$ holds for each integer $m\geq n$.

If $\pi(\xi)=\pi(\eta)=x\in X$, then the following statements hold:
\begin{enumerate}
    \smallskip
    \item[\textnormal{(1)}] $z_m(\xi)\neq z_m(\eta)$ for each $m\geq n$;
    \smallskip
    \item[\textnormal{(2)}] $\lim\limits_{m\to+\infty}z_m(\xi)=\lim\limits_{m\to+\infty}z_m(\eta)=x$;
    \smallskip
    \item[\textnormal{(3)}] $f^n(z_m(\xi))=z_{m-n}(\omega_0)=f^n(z_m(\eta))$;
    \smallskip
    \item[\textnormal{(4)}] $\lim\limits_{m\to+\infty}z_{m-n}(\omega_0)=x_0=f^n(x)$.
\end{enumerate}
Thus we know that $\deg(f^n;x)\geq 2$, which contradicts to $x_0\in {X\setminus \post(f)}$. Since $\xi,\,\eta$ are arbitrary, the map $\pi|_{\sigma^{-n}(\omega_0)}\:\sigma^{-n}(\omega_0)\rightarrow f^{-n}(x_0)$ is injective. Since $x_0\in X\setminus \post(f)$, we have $\card(f^{-n}(x_0))=d^n=\card(\sigma^{-n}(\omega_0))$, and we get that the map $\pi|_{\sigma^{-n}(\omega_0)}$ is bijective.
\end{proof}

The following two lemmas are standard for the full shifts; see for example, \cite{PU10}.

\begin{lemma}\label{l:pressure of shift} 
    Let $\Sigma$ be the space of infinite one-sided sequences of $d$ symbols, and $\sigma\:\Sigma\rightarrow\Sigma$ the shift map.
    Let $\phi$ be a real-valued H\"older continuous function on $\Sigma$. Then for each $\omega_0\in\Sigma$, we have
    \begin{equation*}
    	P(\sigma,\phi)=\lim_{n\to+\infty}
    	\frac{1}{n}\log\sum\limits_{\omega\in \sigma^{-n}(\omega_0)}\exp (S_n\phi(\omega)).
    \end{equation*}
\end{lemma}

\begin{lemma} \label{l:control sum exp S_n for shift}
    Let $\Sigma$ be the space of one-sided infinite sequences on $d$ symbols, and $\sigma\:\Sigma\rightarrow\Sigma$ the shift map.
    Let $\phi$ be an $\alpha$-H\"older continuous function on $\Sigma$. Then there exists $C_0>1$ such that for any $\xi_0,\eta_0\in\Sigma$ and $n\in\N$,
    \begin{equation*}
        \frac{1}{C_0}
        \leq
        \frac
        {\sum_{\xi\in \sigma^{-n}(\xi_0)}\exp (S_n\phi(\xi))}
        {\sum_{\eta\in \sigma^{-n}(\eta_0)}\exp (S_n\phi(\eta))}
        \leq
        C_0.
    \end{equation*}
\end{lemma}

\subsubsection{Characterizations of topological pressure}
We first give some technical lemmas regarding the topological pressure of CXC maps and H\"older continuous potentials.
\begin{lemma}\label{l:pressure_preimages of one point}
    Suppose that $f\:(\fX_1,X)\rightarrow(\fX_0,X)$ and $\{\cU_n\}_{n\in\N_0}$ satisfy the Assumptions in Section~\ref{s:assumption} and the Additional Assumptions in Subsection~\ref{ss:LD addassumptions}. 
    Let $\phi$ be a real-valued H\"older continuous function on $X$. Then for each $x_0\in X\setminus \post(f)$, we have
    \begin{equation*}
        P(f_X,\phi)=\lim_{n\to+\infty} \frac{1}{n}\log\sum\limits_{x\in f^{-n}(x_0)}\exp (S_n\phi(x)).
    \end{equation*}
\end{lemma}
\begin{proof}
    Fix an arbitrary point $x_0\in X\setminus \post(f)$. Since the semiconjugacy $\pi$ in Proposition~\ref{p:semi conjugacy} is surjective, there exists $\omega_0\in\Sigma$ such that $x_0=\pi(\omega_0)$. 
    
    By Lemma~\ref{l:a bijection between sequence and preimages}, $\pi|_{\sigma^{-n}(\omega_0)}$ is bijective, so we can calculate $P(f_X,\phi)$ using Proposition~\ref{p:top pressure} and Lemma~\ref{l:pressure of shift} as follows:
\begin{align*}
    P(f_X,\phi)=P(\sigma,\phi\circ\pi)
    =&\lim_{n\to+\infty}
    \frac{1}{n}\log\sum\limits_{\omega\in \sigma^{-n}(\omega_0)}\exp (S_n(\phi\circ\pi)(\omega))\\
    =&\lim_{n\to+\infty}
    \frac{1}{n}\log\sum\limits_{x\in f^{-n}(x_0)}\exp (S_n\phi(x)).  \qedhere
\end{align*}
\end{proof}

\begin{lemma}\label{l:loc control of Sn}
    Suppose that $f\:(\fX_1,X)\rightarrow(\fX_0,X)$ and $\{\cU_n\}_{n\in\N_0}$ satisfy the Assumptions in Section~\ref{s:assumption} and the Additional Assumptions in Subsection~\ref{ss:LD addassumptions}.
    Let $\phi$ be an $\alpha$-H\"older continuous function on $X$. Then there exists a constant $C_1>0$ such that for each $n\in\N$ and each $U\in\cU_n$, the inequality 
    \begin{equation*}
    	\abs{ S_n\phi(x)-S_n\phi(y) } \leq C_1
    \end{equation*}
    holds for all $ x,\,y\in X\cap U$.
\end{lemma}
\begin{proof}
    Fix arbitrary $n\in\N$, $U\in\cU_n$, and $x,\,y\in U$.

    By Proposition~\ref{p:comparability}~(i), for each integer $i\in [0, n-1]$, we have $\abs{ f^i(x)-f^i(y) }\leq C'd_0\theta^{n-i}$ since $f^i(x),f^i(y)\in f^i(U)\in\cU_{n-i}$. Since $\phi$ is $\alpha$-H\"older continuous, there exists a constant $C>0$, such that $\abs{  \phi(x_1)-\phi(x_2) } \leq C \abs{ x_1-x_2 }^\alpha$ for all $x_1,\,x_2\in X$. So for each integer $i\in [0, n-1]$,
    \begin{equation*}
    	\Absbig{ \phi \bigl( f^i(x) \bigr)-\phi \bigl(f^i(y) \bigr) } \leq C \bigl(C'd_0\theta^{n-i} \bigr)^\alpha.
    \end{equation*}
    Thus
    \begin{equation*}
    \begin{aligned}
            \abs{ S_n\phi(x)-S_n\phi(y) }
            \leq \sum_{i=0}^{n-1}  \Absbig{  \phi \bigl( f^i(x) \bigr)-\phi \bigl( f^i(y) \bigr) } 
            \leq\sum_{i=0}^{n-1}C(C'd_0)^\alpha\cdot(\theta^\alpha)^{n-i}  
            \leq \frac{  C(C'd_0\theta)^\alpha }{  (1-\theta^\alpha) }.
    \end{aligned}
    \end{equation*}
    By putting $C_1\=C(C'd_0\theta)^\alpha /  (1-\theta^\alpha)$, we complete the proof of this lemma.
\end{proof}

\begin{lemma}\label{l:control sum deg exp S_n, locally}
    Suppose that $f\:(\fX_1,X)\rightarrow(\fX_0,X)$ and $\{\cU_n\}_{n\in\N_0}$ satisfy the Assumptions in Section~\ref{s:assumption} and the Additional Assumptions in Subsection~\ref{ss:LD addassumptions}.
    Let $\phi$ be an $\alpha$-H\"older continuous function on $X$.  Then there exists $C_2>1$ such that for all $U_0\in\cU_0$, $x_0,\,y_0\in X\cap U_0$, and $n\in\N$, 
    \begin{equation*}
	C_2^{-1}    \leq    \cS_n(x_0) / \cS_n(y_0)    \leq C_2,
\end{equation*}
where 
\begin{equation}   \label{e:S}
	\cS_n(z) \= \sum_{z'\in f^{-n}(z) }\deg(f^n;z')\cdot \exp (S_n\phi(z'))  
\end{equation}
for each $z \in X$.
\end{lemma}
\begin{proof}
	Fix arbitrary $U_0\in\cU_0$, $x_0,\,y_0\in X\cap U_0$, and $n\in\N$. For each $z \in X$ and each $V \subseteq \fX_1$, write
	\begin{equation*}
		\cS_n(z,V) \= \sum_{z'\in f^{-n}(z) \cap V}\deg(f^n;z')\cdot \exp(S_n\phi(z')).
	\end{equation*}
	Then $\cS_n (z) = \sum_{U\in\cU_n, f^n(U)=U_0}  \cS_n (z,U)$ for each $z\in \{x_0, \, y_0\}$.
    For each $U\in\cU_n$ with $f^n(U)=U_0$, we have
    \begin{equation*}
    	\sum_{x\in f^{-n}(x_0) \cap U}\deg(f^n;x)=\deg(f^n|_U)=\sum_{y\in f^{-n}(y_0) \cap U}\deg(f^n;y),
    \end{equation*}
    since $f^n|_U\:U\rightarrow U_0$ is an FBC map.
    Then it is easy to see from Lemma~\ref{l:loc control of Sn} that 
    \begin{equation*}
        \exp(-C_1)  \leq  \cS_n (x_0,U) /  \cS_n (y_0,U)  \leq    \exp (C_1),
    \end{equation*}
    where $C_1>0$ is the constant from Lemma~\ref{l:loc control of Sn}
    Thus we get
    \begin{equation*}
        \exp(-C_1)
        \leq
        \frac
        {\sum_{U\in\cU_n, f^n(U)=U_0} \cS_n (x_0,U) }
        {\sum_{U\in\cU_n, f^n(U)=U_0} \cS_n (y_0,U) } 
        \leq
        \exp(C_1).
    \end{equation*}
    By putting $C_2\=\exp(C_1)$, we complete the proof of this lemma.
\end{proof}

\begin{lemma}\label{l:control sum deg exp S_n, uniformly}
    Suppose that $f\:(\fX_1,X)\rightarrow(\fX_0,X)$ and $\{\cU_n\}_{n\in\N_0}$ satisfy the Assumptions in Section~\ref{s:assumption} and the Additional Assumptions in Subsection~\ref{ss:LD addassumptions}.
    Let $\phi$ be an $\alpha$-H\"older continuous function on $X$.  Then there exists $C_3>1$ such that for all $x_0, \,y_0\in X$ and $n\in\N$, 
    \begin{equation*}
     C_3^{-1}    \leq    \cS_n(x_0) / \cS_n(y_0)    \leq C_3,
    \end{equation*}
    where $\cS_n(\cdot)$ is defined in (\ref{e:S}) in Lemma~\ref{l:control sum deg exp S_n, locally}.
\end{lemma}
\begin{proof}
    Fix arbitrary $x_0,\,y_0\in X$, $n\in\N$, and $U,\,V\in\cU_0$ such that $x_0\in U$, $y_0\in V$. Since $\post(f)$ is nowhere dense in $X$, there exist 
$u_0,v_0\in X\setminus \post(f)$ with $u_0\in U$ and $v_0\in V$.

By Lemma~\ref{l:control sum deg exp S_n, locally}, we have $C_2^{-1}    \leq    \cS_n(x_0) / \cS_n(u_0)    \leq C_2$ and $C_2^{-1}    \leq    \cS_n(v_0) / \cS_n(y_0)    \leq C_2$ where $C_2>1$ is the constant from Lemma~\ref{l:control sum deg exp S_n, locally}.

Choose $\xi_0,\,\eta_0\in\Sigma$ such that $\pi(\xi_0)=u_0$ and $\pi(\eta_0)=v_0$, where $\pi$ is given in Proposition~\ref{p:semi conjugacy}.
By Lemma~\ref{l:a bijection between sequence and preimages}, we have
\begin{align*}
\cS_n(u_0)=&\sum_{u\in f^{-n}(u_0)}\exp (S_n\phi(u))=\sum_{\xi\in\sigma^{-n}(\xi_0)}\exp(S_n\phi\circ\pi(\xi)),\\
\cS_n(v_0)=&\sum_{v\in f^{-n}(v_0)}\exp (S_n\phi(v)) =\sum_{\eta\in\sigma^{-n}(\eta_0)}\exp(S_n\phi\circ\pi(\eta)).
\end{align*}
Then, by Lemma~\ref{l:control sum exp S_n for shift}, we have $C_0^{-1} \leq \cS_n(u_0)  / \cS_n(v_0) \leq C_0$.
Combining the inequalities above, we have
    \begin{equation*}
    C_0^{-1}\cdot C_2^{-2} \leq \cS_n(x_0)  / \cS_n(y_0) \leq  C_0\cdot C_2^2.
    \end{equation*}
By putting $C_3 \=C_0\cdot C_2^2$, we complete the proof of this lemma.
\end{proof}

A characterization of topological pressure is obtained in the following lemma.

\begin{lemma}\label{l:pressure_sequence in X}
    Suppose that $f\:(\fX_1,X)\rightarrow(\fX_0,X)$ and $\{\cU_n\}_{n\in\N_0}$ satisfy the Assumptions in Section~\ref{s:assumption} and the Additional Assumptions in Subsection~\ref{ss:LD addassumptions}. 
    Let $\phi$ be a real-valued H\"older continuous function on $X$. Then for each sequence $\{x_k\}_{k\in\N}$ of points in $ X$,
    \begin{equation*}
    \begin{aligned}
    P(f_X,\phi)=\lim_{n\to+\infty}
    \frac{1}{n}\log\sum\limits_{x\in f^{-n}(x_n)}\deg(f^n;x)\cdot\exp (S_n\phi(x)).
    \end{aligned}
    \end{equation*}
\end{lemma}
\begin{proof}
    By Lemma~\ref{l:control sum deg exp S_n, uniformly},
    \begin{equation*}
     C_3^{-1}    \leq \cS_n(x') / \cS_n(x'')   \leq C_3
    \end{equation*}
    holds for all $x',x''\in X$, where $C_3>1$ is the constant from Lemma~\ref{l:control sum deg exp S_n, uniformly},
    and $\cS_n(\cdot)$ is defined in (\ref{e:S}) in Lemma~\ref{l:control sum deg exp S_n, locally}. 
    Fix an arbitrary $x_0\in X\setminus \post(f)$. Then for each $n\in\N$, we have
    \begin{equation*}
    	\frac{1}{n}\log\sum\limits_{x\in f^{-n}(x_0)} e^{S_n\phi(x)} -\frac{\log C_3}{n}
    	\leq
        \frac{1}{n}\log ( \cS_n(x_n) )
        \leq
        \frac{1}{n}\log  \sum\limits_{x\in f^{-n}(x_0)} e^{S_n\phi(x)} + \frac{\log C_3}{n} .
    \end{equation*}
    So, by Lemma~\ref{l:pressure_preimages of one point} the limit $\lim_{n\to+\infty}\frac{1}{n}\log (\cS_n(x_n))$ exists and is equal to $P(f_X,\phi)$.
\end{proof}

Now, we are ready to prove our level-$2$ large deviation principle.

\begin{proof}[Proof of Theorem~\ref{t:Intro_LD}]
    Suppose that $f\:(\fX_1,X)\rightarrow(\fX_0,X)$ and $\{\cU_n\}_{n\in\N_0}$ satisfy the Assumptions in Section~\ref{s:assumption} and the Additional Assumptions in Subsection~\ref{ss:LD addassumptions}.

    We apply Theorem~\ref{t:LD, Kifer and Juan} with $Z\=X$, $g\=f_X$, and $H\=C^{0,\alpha}(X)$.    
    By \cite[Theorem~6.8]{He01}, the set of Lipschitz functions on $X$ is dense in $C(X)$, thus $H$ is dense in $C(X)$.

    By Theorem~\ref{t:Intro asym h-expans}, $f_X$ is asymptotically $h$-expansive, thus $h_\centerdot(f_X)$ is upper semi-continuous. Condition~(i) in Theorem~\ref{t:LD, Kifer and Juan} is satisfied.
    By Proposition~\ref{p:Uniqueness equilibrium state}, for each $\psi\in H$, there exists a unique equilibrium state for $f$ and potential $\phi+\psi$. Condition~(ii) in Theorem~\ref{t:LD, Kifer and Juan} is satisfied.
    
    By Lemma~\ref{l:pressure_sequence in X}, for each $\psi\in H$, we have
 \begin{equation*}
 \begin{aligned}
     &\lim_{n\to+\infty}\frac{1}{n}\log\int_{\cP(X)} \! \exp{\biggl(n\int\!\psi \diff\mu\biggr)} \diff\Omega_n(\mu)\\
     &\qquad =
     \lim_{n\to+\infty}\frac{1}{n}\log  \biggl(  \sum\limits_{y\in f^{-n}(x_n)}
    \frac{\deg(f^n;y)\cdot e^{S_n\phi(y)}}
    {\sum\limits_{z\in f^{-n}(x_n)}\deg(f^n;z)\cdot e^{S_n\phi(z)}}e^{\sum_{i=0}^{n-1}\psi(f^i(y))}  \biggr) \\
    &\qquad=
    \lim_{n\to+\infty}\frac{1}{n}
    \biggl(
    \log\sum\limits_{y\in f^{-n}(x_n)} \deg(f^n;y)\cdot e^{ S_n(\phi+\psi)(y) }  -  \log\sum\limits_{y\in f^{-n}(x_n)} \deg(f^n;y)\cdot e^{S_n\phi(y)}
    \biggr)\\
    &\qquad=
    P(f_X,\phi+\psi)-P(f_X,\phi).
\end{aligned}
\end{equation*}

Now that all conditions in Theorem~\ref{t:LD, Kifer and Juan} are satisfied, the result follows from Theorem~\ref{t:LD, Kifer and Juan}.
\end{proof}

We finish this paper with an equidistribution result as a consequence of our level-$2$ large deviation principle. The proof follows from similar arguments in the proof of \cite[Corollary~1.7]{Li15}.

\begin{cor}[Equidistribution]\label{c:Equidistribution}
	Let $f\:(\fX_1,X)\rightarrow(\fX_0,X)$ be a metric CXC system. Assume that the following conditions are satisfied:
	\begin{enumerate}
		\smallskip
		\item[\textnormal{(1)}] $\fX_0$ and $\fX_1$ are strongly path-connected.
		\smallskip
		\item[\textnormal{(2)}] The branch set $B_f$ is finite. 
	\end{enumerate}
	Let $\phi$ be a real-valued H\"older continuous function on $X$, and $\mu_\phi$ be the unique equilibrium state for $f|_X$ and potential $\phi$.
	Let $\{x_n\}_{n\in\N}$ be a sequence of points in $X$. Consider the following sequence of Borel probability measures on $X$:
	\begin{equation*}
		\nu_n\=\sum_{y\in f^{-n}(x_n)}\frac{\deg(f^n;y)\cdot\exp (S_n\phi(y))}{\sum_{z\in f^{-n}(x_n)}\deg(f^n;z)\cdot\exp (S_n\phi(z))}\frac{1}{n}\sum_{i=0}^{n-1}\delta_{f^i(y)},
	\end{equation*}
	where $S_n\phi$ is defined in (\ref{eq:def:Birkhoff average}).
	Then $\nu_n$ converges to $\mu_\phi$ in the weak$^*$ topology as $n\to+\infty$.
\end{cor}



\end{document}